\documentclass[a4paper,reqno,11pt]{amsart}

\usepackage[english]{babel}
\usepackage[utf8x]{inputenc}
\usepackage[T1]{fontenc}

\usepackage[margin=1.5in]{geometry}

\usepackage[alphabetic]{amsrefs}

\usepackage{amsmath}
\usepackage{amstext}
\usepackage{amsfonts}
\usepackage{amssymb}
\usepackage{amsthm}
\usepackage{amsrefs}
\usepackage{mathtools}
\usepackage{dsfont}
\usepackage{color}
\usepackage{graphicx}
\usepackage[colorinlistoftodos]{todonotes}
\usepackage[colorlinks=true, allcolors=blue]{hyperref}
\usepackage{float}
\usepackage{mathrsfs}
\allowdisplaybreaks

\newtheorem{theorem}{Theorem}[section]
\newtheorem{lemma}[theorem]{Lemma}
\newtheorem{prop}[theorem]{Proposition}

\newtheorem{cor}[theorem]{Corollary}
\newtheorem{thm}[theorem]{Theorem}
\newtheorem{lem}[theorem]{Lemma}

\newtheorem{thmintro}{Theorem}

\newtheorem*{cor*}{Corollary}
\newtheorem*{thm*}{Theorem}
\newtheorem*{lem*}{Lemma}

\newtheorem*{prop*}{Proposition}

\theoremstyle{definition}

\newtheorem{defn}[theorem]{Definition}
\newtheorem{example}[theorem]{Example}

\newtheorem*{defn*}{Definition}

\theoremstyle{remark}
\newtheorem{remark}[theorem]{Remark}

\newcommand{\act}{\curvearrowright}

\newcommand{\cA}{\mathcal{A}}
\newcommand{\cB}{\mathcal{B}}
\newcommand{\cC}{\mathcal{C}}

\newcommand{\cG}{\mathcal{G}}
\newcommand{\cH}{\mathcal{H}}

\newcommand{\cK}{\mathcal{K}}
\newcommand{\cL}{\mathcal{L}}

\newcommand{\bC}{{\mathbb{C}}}

\newcommand{\bF}{{\mathbb{F}}}
\newcommand{\bH}{{\mathbb{H}}}

\newcommand{\bN}{{\mathbb{N}}}
\newcommand{\bZ}{{\mathbb{Z}}}

\newcommand{\bR}{{\mathbb{R}}}

\newcommand{\fE}{\mathfrak{E}}

\newcommand{\fI}{\mathfrak{I}}

\newcommand{\dT}{{\mathsf{T}}}

\newcommand{\vf}{{\varphi}}

\newcommand{\Aut}{{\rm Aut}}

\newcommand{\id}{\operatorname{id}}

\newcommand{\G}{\Gamma}

\usepackage[nodisplayskipstretch]{setspace}

\allowdisplaybreaks

\title[]{Operator space complexification transfigured}

\author[D. P. Blecher]{David P. Blecher}
\address{Department of Mathematics\\ University of Houston\\USA}
\email{dpbleche@central.uh.edu}

\author[M. Kalantar]{Mehrdad Kalantar}
\address{Department of Mathematics\\ University of Houston\\USA}
\email{mkalantar@uh.edu}

\begin{document}

\begin{abstract} Given a finite group $G$, a central subgroup $H\le G$, and an operator space $X$ equipped with an action of $H$ by complete isometries, we construct an operator space $X_G$ equipped with an action of $G$ which is unique under a `reasonable' condition. This generalizes the operator space complexification $X_c$ of $X$. 
As a linear space $X_G$ is the space obtained from inducing the representation of $H$ to $G$ (in the sense of Frobenius). \end{abstract} 

\thanks{DB was supported by a Simons Foundation Collaboration Grant.  MK was supported by the NSF Grant DMS-2155162.}

\maketitle



\section{Introduction}
Complexification is a main technical tool in the study of real spaces $X$.  It produces a complex space $X_c$ canonically constructed from $X$, which also shares almost all structural properties with the original real space $X$.
ln particular, a common strategy 
with real spaces is to pass to the complexification. This is especially profitable for structures like $C^*$-algebras where the notion of positivity for instance is a significant aspect of the theory.

The construction of the complexification of $X$ at the algebraic level is very simple:\ $X_c$ is the direct sum of two copies of $X$, the `real' and the `imaginary' parts, and the module action of complex numbers on $X_c$ is canonically defined, similarly to  the construction of  the complex algebra $\bC$ from $\bR$. 
The more interesting component of this process is  the analytical one, namely the norming of the complexification $X_c$. This turns out to be rather non obvious for general spaces.
Obviously one would want to require that the norm of $X_c$ extends the original norm of $X$, equivalently, the embedding of $X\to X_c$ is isometric. Also, it is  `reasonable' to require that the conjugation map $(x, y) \mapsto (x, -y)$  be isometric as well. It turns out that there are in general infinitely many norms on $X_c$ satisfying both conditions for any given real normed space $X$ (take $\ell_p$-norms on the direct sum for instance).

In contrast, given a real operator space $X$, Ruan proved that there is always a unique operator space norm on the complexification $X_c$ such that both the embedding of $X$ into $X_c$ and the conjugation map are completely isometric \cite{RComp}. 
 This complexification is such a powerful tool that  
 it is completely natural to  reflect on its  
primal qualities, and in particular to ask \emph{is there a more general  concept with some potentially similar applications?} 

This is the main concern of this paper. We give an answer to these questions, by offering a novel framework which generalizes the complexification, as well as the less-studied notion of quaternification of real spaces.

At the purely algebraic level, our construction is an induction process in the sense of Frobenius (see e.g.~\cite{Mack}).
Similarly to the case of complexification, the more interesting aspect of the theory is the analytical one.  
 Our inducing is the natural one for the category of operator spaces.

To get into more specifics of our notion, let us begin with the case of constructing $\bC$ from $\bR$. Consider the action of $\bZ_2$ on $\bR$ by change of sign $x\mapsto -x$. Then $\bC = {\rm Ind}_{\bZ_2}^{\bZ_4}(\bR)$ is the space of the induced representation from $\bZ_2$ to $\bZ_4$.
Then, the conjugation map on $\bC$ is the canonical action of the group of automorphisms of $\bZ_4$ who restrict to identity on $\bZ_2$, and the norm of $\bC$ is the one defined by the representation of ${\rm Ind}_{\bZ_2}^{\bZ_4}(\bR)$ by convolution operators.

A similar construction yields the quaternion algebra $\bH$ of from $\bR$; indeed $\bH={\rm Ind}_{\bZ_2}^{Q_8}(\bR)$, and all structural properties of $\bH$ have a natural representation theoretical interpretation.

In our general setup, given a (central) subgroup $H$ of a (finite) group $G$, and certain $H$-operator spaces, which we call $H$-modular (see Definition \ref{modu}), 
we construct a $G$-modular operator space $X_G$, which we call the $G$-{\em ification} of $X$. As a vector space, $X_G$ is canonically isomorphic  to ${\rm Ind}_H^G(X)$, and we prove 
that it admits a unique `reasonable' operator space norm (see Definition \ref{def:conc-Gifr}).  

\begin{thmintro} \label{int}
Assume that $G$ is finite, $H\le Z(G)$ and that the pair $(G, H)$ is admissible (see Definition {\rm \ref{def:erg+good}}).
Then every real or complex faithful $H$-ergodic central $H$-modular operator space $X$ admits a unique reasonable operator space $G$-ification.
\end{thmintro}

This theorem generalizes Ruan's uniqueness theorem from \cite{RComp}, which even in the case of quaternification is new.
On the other hand, the above theorem, from the representation theoretical side,
may also be interpreted  as an imprimitivity theorem in the sense of the theory of induced group representations,
particularly in conjunction with the results in Section 7 below.

We turn to the structure of our paper.   
In Section \ref{Gif} we discuss the $G$-ification at the algebraic level, and prove
some technical results concerning these.   
Some proofs in this part are rather straightforward, as the reader will notice, we include the arguments both for the sake of completeness, and more importantly, to highlight that all our constructions and  steps en route to our analytic setup of Definition~\ref{modu} are completely natural and canonical. 
Then in the second half of Section \ref{Gif} we use the latter mentioned structure to  define a representation of the $G$-ification by means of convolution operators on the induced Hilbert space, and this gives 
 an operator space structure on the $G$-ification.  

In Section \ref{maith}  we prove our main theorem, Theorem \ref{int} above, for `concrete $G$-ifications'. 

In Section \ref{absc} we briefly  give several characterizations of $H$-modular operator spaces $X$. The novel feature here is that we are transferring the concept of operator modules in the sense of  
Christensen-Effros-Sinclair (see e.g.\ \cite[Theorem 3.3.1]{BLM})  to a class of abstract operator spaces with a group action. We have not hitherto seen this in the literature. We also transfer to such spaces a couple of important aspects of the theory of operator modules (e.g.\ relating to operator space multiplier algebras, see Chapters 3 and 4 in \cite{BLM}
  in the complex case (the real case works out almost identically as noted in \cite{BReal}). We then define  abstract  $G$-ifications of modular operator spaces, and we state Theorem \ref{int} above in this setting.

In Section \ref{applns}  we give a matrix representation in $M_m(X)$ of the $G$-ification of $X$ matching Ruan's  observation that the complexification $X_c$  of $X$ is representable completely isometrically as the operator subspace of $M_2(X)$ 
consisting of matrices of the form 
\begin{equation} \label{ofr} \begin{bmatrix} 
       x    & -y \\ 
       y   & x 
    \end{bmatrix} 
    \end{equation} 
    for $x, y \in X$.    
    Thus the $G$-ification norm becomes the usual operator space norm on $M_m(X)$ restricted to a subspace of matrices. 
    This is often enormously helpful in computations with the complexification. 
We then give several applications of the uniqueness theorem
(Theorem~\ref{int}), identifying the operator space $G$-ifications of various spaces related to the original $H$-modular space $X$.
Some of these are generalizations of facts about the complexification, 
others generalize facts from the theory of induced representations.

Section~\ref{sec:mod} is devoted to the discussion of the extended module action on the $G$-ification. We should note that the primary objective of the complexifying a real space $X$ is to obtain a space on which $\bC$ acts canonically. Put in our framework, this is the canonical module action given by the span of operators defined by the $\bZ_4$-action on $\bR_{\bZ_4}=\bC$. 
In the general case of $G$-ifications of $H$-modular spaces, we will face several natural choices for extending this fact, and norming the algebra that acts on the $G$-ification. We prove that all of these choices yield the same $C^*$-algebra.

Section~\ref{sec:Gified} is concerned with the question of for a given $H\le Z(G)$, which $G$-modular operator spaces $Y$ are $G$-ifications of $H$-modular operator spaces $X$? 
These may be regarded as analogues of the classical imprimitivity theorem.

Finally in Section~\ref{sec:ex} we discuss several concrete examples of $G$-ifications. 

For background and preliminary material used in this work, we refer the reader to \cite{BLM,ER,Pisbk} for the operator space aspects, and for real operator spaces 
and their complexification we refer to \cite{RComp,BReal} and references therein.

\section{The $G$-ification} \label{Gif} 
In this section we introduce and study the notion of $G$-ification $X_G$ of an
operator space  $X$, where $G$ is a group containing $H$ as a subgroup.
 At the basic algebraic level, the $G$-ification $X_G$ as a vector space coincides with the induced space ${\rm Ind}_H^G(X)$. 
The main structural aspect of $X_G$ is 
its operator space structure, and also the order structure in the case where $X$ is an operator system or a $C^*$-algebra.
However we begin with the the linear structure of the $G$-ification before turning our attention to the analytical side of the theory.

\subsection{The algebraic constructions} 
Let $X$ be a vector space over a field $\bF$ (which is usually $\bR$ or $\bC$). Let $H$ be a group of invertible linear maps on $X$. Let $G$ be a group containing $H$ as a subgroup.
\subsubsection{Linear structure} \label{Linstr}  

We begin with many definitions and notation that are used in the paper.
We first recall the notion of \emph{induced representations} in the sense of Frobenius et al.
Let 
\[
X_G:=\{\vf: G\to X \mid \vf(gh) = h^{-1}\vf(g) ~~ \forall g\in G, h\in H\} .
\]
(This is often written as ${\rm Ind}_H^G(X)$ in the literature.) Then $G$ acts on $X_G$ by left translations:
\[
(g\cdot \vf) (k) := \vf(g^{-1}k) ~~~~~~~ \forall g, k\in G, \vf\in X_G .
\]

So, each $g\in G$ defines an invertible linear map $\pi_X(g)$ on $X_G$, and the map $\pi_X$ is a representation of $G$ on $X$. 

If $H\le Z(G)$, then $G$ also acts on $X_G$ by right translations $(\vf\cdot g) (k) := \vf(kg)$, 
$g, k\in G$, $\vf\in X_G$.  This  yields a representation  $\pi_X^r$ 
of $G$ on $X_G$ commuting with $\pi_X$.   

Denote $\cG^\ell:={\rm span}\{\pi_X(g) : g\in G\}$. Then $\cG^\ell$ is an algebra of maps on $X_G$, and $X_G$ is a left $\cG^\ell$-module.
Similarly, if $H\le Z(G)$, we have the algebra $\cG^r:={\rm span}\{\pi_X^r(g) : g\in G\}$, and $X_G$ is a $\cG^\ell$-$\cG^r$-bimodule if 
$\cG^r$ is given its reversed multiplication.
We denote these module actions by $\alpha\cdot \vf$ and $\vf\cdot\beta$ for $\alpha\in \cG^\ell$, $\beta\in \cG^r$, and $\vf\in X_G$.
We also denote by $\mathsf{J}:\cG^\ell\to\cG^r$ the linear extension of the map $\pi_X(g)\to \pi_X^r(g^{-1})$. We will see later that this is
a well defined $*$-isomorphism. 

Given $x\in X$, define $\vf_x: G\to X$ by
\[
\vf_x(g) := \left\{
\begin{array}{lr}
        g^{-1}x & \text{if } g\in H\\
        0 & \text{if } g\notin H 
    \end{array}
\right.
\]
Then $\vf_x\in X_G$ and the map $\fE_X \colon x\mapsto \vf_x$ is an $H$-equivariant $\bF$-linear embedding of $X$ into $X_G$.

Note that if $\pi_X(g)(\vf_x) = \vf_x$ for an element $g\in G$ and a non-zero $x\in X$, then $g\in H$ and $gx = x$. In particular, the representation $\pi_X$ of $G$ is faithful.

Let $X$ and $Y$ be vector spaces with faithful actions of $H$ by invertible linear maps.
We say a linear map $T: X\to Y$ in an \emph{$H$-intertwiner} if $T(hx) = hT(x)$ for all $h\in H$ and $x\in X$.
We denote by $\cL(X, Y)^H$ the space of all linear $H$-intertwiners from $X$ to $Y$. Note that $\cL(X)^H$ is the commutant of $H$ in $\cL(X)$, and is a subalgebra of $\cL(X)$.
If $T: X\to Y$ is an $H$-intertwiner, then for every $\vf\in X_G$, $T\circ \vf \in Y_G$.
Indeed, for every $g\in G$ and $h\in H$, we have
\[
T \big(\vf(gh)\big) = T \big(h^{-1}\vf(g)\big) = h^{-1}T \big(\vf(g)\big) .
\]
Thus, we get a map $T_G : X_G \to Y_G$ defined by
\begin{equation}\label{def:map-Gif}
T_G(\vf) := T\circ \vf ~~~~~ (\vf\in X_G),
\end{equation}
which we call the $G$-{\em ification} of $T$.

Denote $\Aut_H(G)=\{\sigma\in\Aut(G) : \sigma(h) = h \text{ for all } h\in H\}$ for the group of all automorphisms of $G$ that restrict to the identity on $H$.

We observe that for every $\vf\in X_G$ and $\sigma\in \Aut_H(G)$, $\vf\circ \sigma\in X_G$.
In particular, we have a (left) action $\Aut_H(G)\act X_G$ given by $\sigma\cdot \vf := \vf\circ \sigma^{-1}$. 
This action plays a crucial role in determining the operator space structure of the $G$-ification.

{
\begin{example}[{\bf Complexification}]\label{ex:comp}
Let $H=\bZ_2$ and $G=\bZ_4$. 
Let $X$ be a real vector space, and consider the action of $H$ on $X$ defined by $x\mapsto -x$. Then $X_G \cong X_c$ is the complexification of $X$ via the identification $X_G\ni \vf \mapsto \vf(0) + i \vf(1) \in X_c$.

The algebra $\cG^\ell= \cG^r$ is the algebra $\bC$, and the module actions are the scalar product of $\bC$.

A map $T:X\to Y$ between real vector spaces is $H$-intertwiner if and only if it is $\bR$-linear. In this case the map $T_G:X_G\to Y_G$ coincides with the complexification $T_c: X_c\to Y_c$ of $T$ defined by $T_c(x+iy) = T(x)+iT(y)$.

We have $\Aut_H(G)=\{\id, \sigma\}$, where $\sigma(1)=3$. The action of $\sigma$ on $X_G$ corresponds to the conjugation map $x+iy \mapsto x-iy$ on $X_c$.
\end{example}


\begin{example}[{\bf Quaternification}]\label{ex:quat}
Let $H=\bZ_2$ and $G=Q_8$. 
As in the previous example, let $X$ be a real vector space, and consider the action of $H$ on $X$ defined by $x\mapsto -x$. 
Then $X_G = X_\bH$ is the quaternification of $X$ (see \cite{Ng}) via the identification $X_G\ni \vf \mapsto \vf(0) + i \vf(i) + j\vf(j) + k \vf(k)\in X_\bH$.

The algebra $\cG^\ell=\cG^r$ is the Hamilton quaternions $\bH$, and the module actions are the 
usual product of $\bH$.

A map $T:X\to Y$ between real vector spaces is $H$-intertwiner if and only if it is $\bR$-linear. In this case the map $T_G:X_G\to Y_G$ coincides with the quaternification $T_\bH: X_c\to Y_c$ of $T$ defined by $$T_\bH(x+iy+jw+kz) = T(x)+iT(y)+jT(w)+kT(z).$$
We have $\Aut_H(G)=\Aut(G)\cong S_4$ in this (quaternionic) situation. \end{example}
}


\begin{defn}\label{def:erg+good}
We say the action of $H$ on $X$ is \emph{ergodic} (or that $X$ is $H$-ergodic) if 0 is the unique $H$-fixed vector in $X$.

We say the pair $(G, H)$ is \emph{admissible} if there is a subgroup $\G$ of $\Aut_H(G)$ such that $H\subseteq \{g^{-1}\sigma(g) : \sigma\in \G\}$ for every $g\in G\setminus H$.
In this case, we say that the pair $(G, H)$ is $\G$-{admissible}.
\end{defn}
Obviously, for any real vector space $X$, the action $\bZ_2\act X$, $x\mapsto -x$ is ergodic. Also, for any complex vector space $X$, and any $n\in \bN$, the action $\bZ_n\act X$, $k\cdot x := e^{\frac{2k\pi i}{n}}x$ is ergodic.

We thank Nico Spronk for suggesting a link between our admissible criterion, and the commutator subgroup of $G$. Indeed if $\Gamma$ contains the inner automorphisms then 
$\{g^{-1}\sigma(g) : \sigma\in \G\}$ contains the commutators $[g,k]$ for $g, k \in G$.  

Note that if the pair $(G, H)$ is $\G$-admissible for some $\G\le \Aut_H(G)$, then it is also $\Aut_H(G)$-admissible. But in practice, sometimes a smaller $\G$ suffices to guarantee admissibility. For instance, the pair $(Q_8, \bZ_2)$ is $\G$-admissible, where $\G={\rm Inn}(Q_8)$ is the subgroup of inner automorphisms of $Q_8$. The group ${\rm Inn}(Q_8)$ has order 4, while $\Aut(Q_8)$ has order 24.
 
One can see that the following pairs are $\Aut_H(G)$-admissible, by a simple inspection of the automorphism groups in each case: $(\bZ_{2^n}, \bZ_2)$ for every $n\in\bN$, 
 $(D_4, Z(D_4))$, and 
$(\bZ_{p^2}, \bZ_p)$ for any prime $p$. 
On the other hand, the pairs $(\bZ_6, \bZ_2)$ and $(\bZ_6, \bZ_3)$ are not admissible. More generally, if $m, n\in \bN$ are such that $\gcd(m+1, n) \neq 1$, then the pair $(\bZ_{mn}, \bZ_m)$ is not admissible.

We use the following fact in the proof of our main result Theorem \ref{thm:Gif-unq}.

\begin{lem}\label{lem:sumzero}
Suppose $H$ is a normal subgroup of $G$ such that $(G, H)$ is $\G$-admissible for a finite subgroup $\G$ of $\Aut_H(G)$.
 If $H$ acts ergodically on $X$, then for every $g\in G\setminus H$,
\[
\sum_{\sigma\in \G} \pi_X(\sigma(g)) = 0 .
\]
\end{lem}
\begin{proof}
We observe that since $(G, H)$ is $\G$-admissible, the set $\{\sigma(g) : \sigma\in \G\}$ is $H$-invariant for every $g\in G\setminus H$.  Indeed, since $\sigma(g)  \in G\setminus H$ we have $\sigma(g)h =  \sigma'(g)$ for some $\sigma' \in \G$.

Moreover, for any fixed $g\in G$ and $k\in \{\sigma(g) : \sigma\in \G\}$, we have $\#\{\sigma\in \G : \sigma(g) = k\} = |\G\cap \Aut_{H\cup\{g\}}(G)|$. It is then 
easy to see that for any $h, h^\prime\in H$, $\{\sigma(g)h\}_{\sigma\in \G}$ and $\{\sigma(g)h^\prime\}_{\sigma\in \G}$ are the same list of elements appearing exactly the same number of times.

Thus, it follows for every $g\in G\setminus H$, $h\in H$, $k\in G$, and $\varphi\in X_G$,
\begin{align*}
&\sum_{\sigma\in \G} h\cdot[\big(\pi_X(\sigma(g))(\varphi)\big)(k)] = 
\sum_{\sigma\in \G} \varphi(\sigma(g)^{-1}kh^{-1}) \\=&
\sum_{\sigma\in \G} \varphi(\sigma(g^{-1})(kh^{-1}k^{-1})k) = 
\sum_{\sigma\in \G} \varphi(\sigma(g^{-1})k) , 
\end{align*}
which shows that $\sum_{\sigma\in \G} \big(\pi_X(\sigma(g))(\varphi)\big)(k)$ is an $H$-fixed point in $X$, hence is zero by the ergodicity assumption.
\end{proof}

For $\G\le\Aut_H(G)$, let $(X_G)^{\G}:= \{\vf\in X_G \mid \vf\circ\sigma=\vf ~\text{for all}~ \sigma\in \G\}$ be
the space of $\G$-fixed points of $X_G$.

\begin{prop} \label{ifgFP} 
Assume that the action of $H$ on $X$ is ergodic and the pair $(G, H)$ is $\G$-admissible for some $\G\le \Aut_H(G)$. Then
\[
\fE_X(X) = (X_G)^{\G} .
\]
\begin{proof}
Let $x\in X$, $\sigma \in \G$, and let $g\in G$. If $g\in H$, then $(\vf_x\circ\sigma) (g) = \vf_x(g)$ since $\sigma$ restricts to identity on $H$, and if $g\notin H$ then $\sigma(g)\notin H$ and hence $\vf_x(\sigma (g)) = 0 = \vf_x(g)$ by the definition of $\vf_x$. This shows $\fE_X(X) \subseteq (X_G)^{\G}$.

Conversely, assume $\vf\in (X_G)^{\G}$, and let $x_0 = \vf(e)$. Then, for every $h\in H$ we have
\[
\vf(h) = \vf(eh) = h^{-1}\vf(e) = 
\vf_{x_0}(h) .
\]
On the other hand, if $g\in G\setminus H$, then since the pair $(G, H)$ is $\G$-admissible, for each $h \in H$ there exists $\sigma\in \G$ such that $\sigma(g) = gh$. Then,
\[
h^{-1} \vf(g) = \vf(gh) = \vf(\sigma(g)) = \vf(g) .
\]
This shows that $\vf(g)$ is an $H$-fixed element of $X$, hence $\vf(g)= 0$ by the ergodicity assumption. Thus, $\vf = \vf_{x_0}$.
\end{proof}
\end{prop}

{

\begin{example} In the case of complexification of a real space $X$, that is when $H=\bZ_2$ and $G=\bZ_4$, $\fE_X(X)$ corresponds to the canonical embedding of $X$ into $X_c$, $x\mapsto x+i0$. 
And Proposition~\ref{ifgFP} is the fact that the range of the latter embedding is the fixed point space of the canonical period 2 conjugate linear automorphism on $X_c$.

Similar observations hold in the quaternification case.
\end{example}

}


\subsubsection{Involutions}

In the setup of the previous section, we further assume that $X$ has an involution $\ast$, that is a period-2 automorphism of $X$.
Then there is a canonical involution on $X^G$, the space of all functions from $G$ to $X$, defined by $\vf^\ast(g) := \vf(g^{-1})^\ast$ for all  $\vf\in X^G$ and $g\in G$. 
The subspace $X_G$ is not in general self-adjoint.

\begin{defn}\label{def:adj}
We say that the left action of $H$ on $X$ is \emph{left-involutive} 
if $(h\cdot x)^* = h^{-1}\cdot x^*$ for every $h\in H$ and $x\in X$.  
\end{defn}


\begin{prop}\label{prop:X_G-selfadj}
Assume the action of $H$ on $X$ is \emph{left-involutive}.
If $H\le Z(G)$ then the formula $\vf^\ast(g) := \vf(g^{-1})^\ast$ 
defines an involution on $X_G$. In this case, the map $\fE_X \colon X\to X_G$ is $*$-linear. \begin{proof}
Let $\vf\in X_G$, $g\in G$ and $h\in H$. We have
\begin{align*}
\vf^\ast(gh) &= \vf(h^{-1}g^{-1})^\ast = \vf(g^{-1}h^{-1})^\ast 
\\&= \big(h\cdot\vf(g^{-1})\big)^\ast = h^{-1}\cdot\vf(g^{-1})^\ast = h^{-1}\cdot\vf^\ast(g) .
\end{align*}
This implies $\vf^*\in X_G$. The mapping $\vf\mapsto \vf^*$ is then obviously an involution on $X_G$ and so the first claim follows.

For $x\in X$, $g\in G$ we have
\begin{align*}
\fE_X(x^*)(g) &= \left\{
\begin{array}{lr}
        g^{-1}x^* & \text{if } g\in H\\
        0 & \text{if } g\notin H 
    \end{array}
\right. 
= \left\{
\begin{array}{lr}
        (gx)^* & \text{if } g\in H\\
        0 & \text{if } g\notin H 
    \end{array}
\right. 
\\&= \big(\fE_X(x)(g^{-1})\big)^*= \big(\fE_X(x)\big)^*(g) ,
\end{align*} so that $\fE_X$ is $*$-linear. \end{proof}
\end{prop}

\begin{lem}\label{lem:Gif-map-*}
Assume $H\le Z(G)$. Let $X$ and $Y$ be involutive vector spaces equipped with involutive actions of $H$, and let $T: X\to Y$ be an {$H$-intertwiner} $*$-map. Then $T_G$ is also a $*$-map.
\begin{proof}
Let $\vf\in X_G$. We have 
\begin{align*}
T_G(\vf^*)(g) &= T(\vf^*(g)) = T(\vf(g^{-1})^*) = T(\vf(g^{-1}))^* \\&= \big(T_G(\vf) (g^{-1})\big)^*= \big(T_G(\vf)\big)^* (g) ,
\end{align*} for every $g\in G$.
\end{proof}
\end{lem}

\subsubsection{Algebra structure}

In this section we assume that $G$ is finite.  If $X$ is an algebra over $\bF$, then there is a natural ``convolution'' product on the set of all functions from $G$ to $X$. But in general, when $H$ acts by linear invertible maps on $X$, and $H\le G$, the space $X_G$ is not closed under the convolution. 
To guarantee that, we need to assume some compatibility condition between the $H$-action and the multiplication of $X$, as follows.

\begin{defn}\label{def:H-centric}
Let $X$ be an algebra with an $H$-action by invertible linear maps. 
We say that $X$ is an \emph{$H$-centric} algebra if 
 $H$ acts on $X$ by left-$X$-module maps, that is
\begin{equation}\label{def:mod-act}
~~~~~h \cdot (xy) = x\, (h\cdot y) ~~~~~~ (\text{for all }\ h\in H, \ x, y\in X) .
\end{equation}
\end{defn}

\begin{lem}\label{lem:X_G-alg}
Let $X$ be an $H$-centric algebra over $\bF$. Then the convolution product 
\begin{equation}\label{def:conv-prod}
\vf*\psi(g) : = \frac{1}{|H|}  \, \sum_{k\in G} \vf(k)\psi(k^{-1}g) ~~ (\vf, \psi \in X_G,\, g\in G)
\end{equation}
turns $X_G$ into an algebra. In this case, the map $\fE_X \colon X\to X_G$ is a homomorphism.

Furthermore, $\alpha\cdot(\vf*\psi) = (\alpha\cdot\vf)*\psi$ for every $\vf, \psi \in X_G$, $\alpha\in \cG^\ell$, and if $H\le Z(G)$ then also $(\vf*\psi)\cdot\beta = \vf*(\psi\cdot\beta)$ for every $\beta\in \cG^r$.
\begin{proof}
Assume that $X$ is an $H$-centric algebra. 
Then, for every $g\in G$ and $h\in H$ we have
\begin{align*}
\vf*\psi(gh) 
&=  \frac{1}{|H|} \, \sum_{k\in G} \vf(k)\psi(k^{-1}gh) 
=  \frac{1}{|H|} \, \sum_{k\in G} \vf(k)(h^{-1}\psi(k^{-1}g)) 
\\&= \frac{1}{|H|} \,  \sum_{k\in G} h^{-1}\big(\vf(k)\psi(k^{-1}g)\big) 
= h^{-1}(\vf*\psi(g)) ,
\end{align*}
which shows $\vf*\psi\in X_G$. 

It is straightforward to check, as in the classical case when $X$ is scalar, that this multiplication is associative and turns 
$X_G$ into an algebra over $\bF$.

For the last assertion, we observe
\begin{align*}
[\pi_X(g)(\vf*\psi)](g') 
&=  \frac{1}{|H|} \, \sum_{k\in G} \vf(k)\psi(k^{-1}g^{-1}g') 
\\&=  \frac{1}{|H|} \, \sum_{k\in G} \vf(g^{-1}k)\psi(k^{-1}g')
= [(\pi_X(g)\vf)*\psi](g') ,
\end{align*}
and similar computations show the right module claim.
\end{proof}
\end{lem}

\begin{remark}
If $H$ is abelian, then conversely to Lemma~\ref{lem:X_G-alg}, one can see that if $X_G$ is closed under the convolution product~\eqref{def:conv-prod} then $X$ is $H$-centric.
\end{remark}

\subsection{Convolution operators and norming the $G$-ification} \label{coan} 

We now turn our attention to the norm structure of the $G$-ification of an operator space $X$ endowed with an action of a finite group $H$ by complete isometries.
Given a completely isometric representation $X\subset B(\cH)$ for a Hilbert space $\cH$, our goal is to construct from that, in a canonical way, a representation of $X_G$
on the Hilbert space $\cH_G$, where the latter is endowed with the Hilbert space norm inherited from $\ell^2(G, \cH)$.    
This will be  done by means of convolution maps.  This representation endows  $X_G$ with an operator space structure. 

Throughout the section, we assume $G$ is finite and $H$ is an abelian subgroup of $G$.  One effect of this is the following fact.
\begin{lem}\label{lem:pol-id}
Let $G$ be a finite group and $H$ an abelian subgroup of $G$. Then
\[
X_G = \sum_{g \in G} \, \pi_X(g) \, \fE_X(X) 
= \oplus_{k=1}^m \, \pi_X(g_k) \, \fE_X(X) ,
\]
where $\{ g_i : i = 1, \cdots , m \}$ is a full set of mutually inequivalent representatives of the coset space $G/H$.
\begin{proof}
Let $\{ g_i : i = 1, \cdots , m \}$ be a full set of mutually inequivalent representatives of $G/H$. Then for every $\vf\in X_G$ we have  $\varphi = \sum_k \, \pi_X(g_k) \, \fE_X(\varphi(g_k))$. To see this simply note that $(\pi_X(g_k) \, \fE_X(\varphi(g_k)))(g_j) = 0$ if $k \neq j$, and is $\varphi(g_k)$ if $k = j$.
So $\varphi = \sum_k \, \pi_X(g_k) \, \fE_X(\varphi(g_k))$ on each $g_j$, hence on $G$.   To show that the above sum is indeed a direct sum, assume $x_1, x_2, \dots, x_m \in X$ are such that 
$\varphi = \sum_{k=1}^m \, \pi_X(g_k) \, \fE_X (x_k) = 0$.  Evaluating at $g_j$ as above we have $x_k = 0$ for all $k$, and it follows that $\vf=0$.
\end{proof}
\end{lem}
There is a canonical idempotent map on $X_G$ with range $\fE_X(X)$, namely $E(\varphi ) =  \fE_X(\varphi(1))$.   Note that 
$E( g \,  \fE_X(x)) = 0$ for $x \in X$ and $g \notin H$, since $\fE_X(x)(g^{-1}) = 0$.   
Also,
$E( g_k^{-1} \varphi) = \fE_X(\varphi( g_k)).$ It follows that 
$$\varphi = \sum_k \, g_k \, E( g_k^{-1} \varphi) , \qquad  \varphi  \in X_G .$$
We call this the {\em polarization identity}.

\begin{defn} \label{modu}
Let $H$ be a group. By a {\em left $H$-modular operator space} we mean an operator space $X$ (real or complex), equipped with a left action of $H$ by complete isometries such that there is a completely isometric representation $\rho: X\to B(\cH)$ and a unitary representation $\theta: H \to B(\cH)$ such that $\rho(h x) = \theta(h) \rho(x)$ for all $h\in H$ and $x\in X$.
We call a triple $(\theta, \rho, \cH)$ as in the above an \emph{$H$-modular representation} of $X$.
Right $H$-modular operator spaces and representations are defined similarly, and an  {\em $H$-modular operator space} is simultaneously left and right modular.  
 We say that such 
$X$ is a {\em central $H$-modular operator space}  if in addition $\rho(h x) = \theta(h) \rho(x)  = \rho(x)  \theta(h)$ for all $h\in H$ and $x\in X$.    

We say that $X$ is faithful if the action of $H$ on $X$ is faithful.  This implies that the representation $\theta$ is faithful. 
\end{defn}
In Section \ref{chmod} we will give alternative abstract characterizations of  the above 
kinds of $H$-modular operator spaces.  We shall also see in Theorem \ref{ces} there that representations of a  $H$-modular operator space may be chosen to be triples $(\theta, \rho, \cH)$ as  above
with $\rho(h x) = \theta(h) \rho(x)$  and $\rho(x h ) =  \rho(x) \theta(h)$ for $h\in H$ and $x\in X$.  Henceforth whenever we write such a triple for a  $H$-modular operator space representation we shall
assume that these last equations hold.

The above gives an action of $H$ on the Hilbert space $\cH$ defined by $h\cdot \xi := \theta(h) \xi$, and an action of $H$ on $\cB(\cH)$ defined by $h\cdot T := \theta(h) T$ 
for every $h\in H$, $\xi\in \cH$ and $T\in \cB(\cH)$.

Let $\cC=\theta(H)^\prime$ be the commutant of $\{\theta(h) : h\in H\}$ in $\cB(\cH)$.  
If $H$ is abelian, then 
$\cC$ is invariant under the above $H$-action, and the von Neumann algebra $\cC$ 
turns into an involutive $H$-centric algebra and a central $H$-modular operator space. We will see that $\cC_G$ is an $\Aut_H(G)$-$C^*$-algebra 
equipped with a linear $G$-action
by linear complete isometries (see 
e.g.\ Lemma~\ref{lem:Aut-cov-rep} below).

We construct a canonical representation of $\cC_G$ on the Hilbert space $\cH_G$ by means of convolution operators.  Following the above notation, 
for every $\Psi\in \cC_G$ and $\xi\in \cH_G$, define $\fI(\Psi) (\xi) : G\to \cH$ by
\[
[\fI(\Psi) (\xi)](g) := \frac{1}{|H|} \sum_{k\in G} \Psi(k) \big(\xi(k^{-1}g)\big) .
\]

The following fact, which we use in several places later in the paper, follows by straightforward calculations, and we omit the details.
\begin{lem}\label{lem:fIfE=fT}
Assume $H\le Z(G)$. Then  
$\fI(\fE(T)) = T_G$
for every $T\in \cC$, where $T_G$ is the $G$-ification of $T$ defined in~\eqref{def:map-Gif} as a map on $\cH_G$. \end{lem}

\begin{prop}\label{prop:conv-rep}
Following the above notation, 
for every $\Psi\in \cC_G$ and $\xi\in \cH_G$ we have that  $\fI(\Psi) (\xi) \in \cH_G$, and the map $\fI(\Psi) : \cH_G\to \cH_G$ is bounded and linear.
Furthermore, the map 
$\fI : \cC_G\to \cB(\cH_G)$ is an injective $*$-homomorphism, and it is $G$-equivariant, that is, $\fI(\pi_\cC(g) \, \Psi) = \pi_\cH(g) \fI(\Psi)$ for $g\in G$, and $\Psi\in \cC_G$. \end{prop}	
\begin{proof}
Let $\Psi\in \cC_G$ and $\xi\in \cH_G$. Then for every $g\in G$ and $h\in H$ we have
\begin{align*}
[\fI(\Psi)(\xi)](gh) 
&= \frac{1}{|H|} \sum_{k\in G} \Psi(k) \big(\xi(k^{-1}gh)\big) 
\\&= \frac{1}{|H|} \sum_{k\in G} \Psi(k) \big(\theta(h^{-1})\xi(k^{-1}g)\big) 
\\&= \frac{1}{|H|} \sum_{k\in G} \theta(h^{-1})\left(\Psi(k) \big(\xi(k^{-1}g)\big)\right) 
\\&= \theta(h^{-1})\big([\fI(\Psi)(\xi)](g)\big) ,
\end{align*}
and therefore $\fI(\Psi)(\xi) \in \cH_G$.

It is obvious that $\fI(\Psi) : \cH_G\to \cH_G$ is linear.  To see that it is bounded note that by 
Lemma~\ref{lem:pol-id} this reduces to showing that $\fI(\fE(T)))$ is bounded for $T \in \cC$.  This follows from 
 Lemma \ref{lem:fIfE=fT} since $T_G(\xi) = T \circ \xi$.

That $\fI$ is a $*$-algebra homomorphism follows by straightforward calculations
that most readers will be familiar with from the setting of group convolutions.
To see that $\fI$ is one-to-one suppose that $[\fI(\Psi)(\fE_\cH(x))](g) = 0$ for some $\Psi \in \cC_G$ and for all
$x \in \cH, g \in G$. Then 
\[
0 = \sum_{k\in G} \, \Psi(k) \, \big(\fE_\cH(x))(k^{-1} g)\big) = \sum_{h \in H}  \Psi(gh) \, \big( \theta(h) x\big) = |H|\Psi(g)(x).
\]
Thus $\Psi(g) = 0$, and $\Psi = 0$.

The  $G$-equivariance is similar to the proof of Lemma \ref{lem:X_G-alg}:  $\fI(\pi_\cC(g) \, \Psi) (\xi)(g')$ is equal to 
$$\frac{1}{|H|}\sum_{k\in G} (\pi_\cC(g) \, \Psi)(k)\big(\xi(k^{-1}g')\big) = \frac{1}{|H|}\sum_{k\in G} \, \Psi(g^{-1} k)\big(\xi(k^{-1}g')\big),
$$  for $g, g' \in G$, and $\Psi\in \cC_G$. 
Setting $r = g^{-1} k$ we obtain 
$$\frac{1}{|H|}\sum_{k\in G} \, \Psi(r)\big(\xi(r^{-1} g^{-1} g')\big)  = (\pi_\cH(g) \fI(\Psi)) (\xi)(g').$$ 
Thus  $\fI(\pi_\cC(g) \, \Psi) = \pi_\cH(g) \fI(\Psi)$. 
\end{proof}

\begin{lem}\label{lem:red-reg-*}
Assume $H\ \le Z(G)$. Then
the map $\fE_\cC: \cC\to \cC_G$ is a unital 
injective $*$-homomorphism.

\begin{proof}
That $\fE_\cC$ is one-to-one is obvious from the definition of the map, and that it is a $*$-map was shown in Proposition~\ref{prop:X_G-selfadj}.
Since $\fI$ is an injective homomorphism (Proposition~\ref{prop:conv-rep}), and $(ST)_G = S_G T_G$ for all 
$S, T\in \cC$, Lemma~\ref{lem:fIfE=fT} yields that $\fE_\cC$ is a homomorphism.
It is an exercise that $\fE_\cC$ is unital. 
\end{proof} \end{lem}

The proof of the following identity is straightforward, we omit the details.
\begin{lem}\label{lem:Aut-cov-rep}
For every $\sigma\in \Aut_H(G)$, the map $\xi\mapsto\xi\sigma^{-1}$ defines a unitary $u_\sigma$ on ${\cH}_G$.
If $H$ is abelian, then for every $\sigma\in \Aut_H(G)$ and $\Psi \in \cC_G$, we have
\[
\fI(\Psi \sigma^{-1}) = u_\sigma \fI(\Psi) u_{\sigma^{-1}} .
\]
\end{lem}

\section{Characterization of the $G$-ification} \label{maith}

\begin{defn}\label{def:conc-Gif}
Let $X$ be a 
central $H$-modular operator space, and $H$ a subgroup of $G$.
By a {\em concrete operator space $G$-ification of $X$} we mean a quadruple $(V,\pi, \Pi, \cK)$, where 
$\pi:G\to \cB(\cK)$ is a unitary representation of $G$ on a 
Hilbert space $\cK$, and $\Pi: X\to V \subset \cB(\cK)$ is a completely isometric representation of $X$ such that 
\begin{itemize}
\item[(i)]
$\Pi(X)$ is contained in the commutant of $\pi(G)$.
\item[(ii)] 
$\Pi(hx) = \pi(h) \Pi(x)$ for all $x\in X$ and $h\in H$.
\item[(iii)] 
The map $\fE_X(x)\pi_X(g)\mapsto \Pi(x) \pi(g)$ for $x\in X$ and $g\in G$, extends to a vector space isomorphism $\tilde\Pi:X_G\to V \subset B(K)$. 
\end{itemize} 
\end{defn} 

For brevity we sometimes also refer to $V = \tilde\Pi(X_G)$ as the concrete operator space $G$-ification of $X$, and we note that by Lemma~\ref{lem:pol-id},  condition (iii) 
is saying (assuming $H$ abelian) that 
$$V = \tilde\Pi(X_G) = {\rm span}\{\Pi(x) \pi(g) : x\in X,  g\in G\} = \oplus_{k=1}^m \,   \Pi(X) \, \pi(g_k),$$
where $\{ g_k : k = 1, \cdots , m \}$ is a full set of mutually inequivalent representatives of $G/H$. 
This is analogous to the fact that for a a real operator space $X$,
a concrete complexification of $X$ is  $\Pi(X) + i \Pi(X) = \Pi(X) \oplus i \Pi(X)$ inside $B(\cK)$, for a 
real  complete isometry
$\Pi : X \to B(\cK)$, for a complex Hilbert space $\cK$.  

\begin{defn}\label{def:conc-Gifr} Given $\G\le \Aut_H(G)$, we say the $G$-ification $(V,\pi, \Pi, \cK)$ is \emph{$\G$-reasonable} if the action of $\G$ on $\tilde\Pi(X)$ is by complete isometries,
that is, $\|\tilde\Pi(\vf\sigma)\| = \|\tilde\Pi(\vf)\|$
for every $\sigma\in\G$ and $\vf\in X_G$, and the same equality holds at each matrix level. 
\end{defn}
%
The reasonable condition means 
that 
$$\|\sum_{k=1}^m\Pi(x_k)\pi(g_k)\| = \|\sum_{k=1}^m\Pi(x_k)\pi(\sigma(g_k))\|$$ for every $\sigma\in \G \subset \Aut_H(G)$ and $x_1, \dots, x_k\in X$
(and similarly at each matrix level).

We now prove our main result.

\begin{thm} \label{thm:Gif-unq}
Assume $H\le Z(G)$, $G$ is finite, and that the pair $(G, H)$ is $\G$-admissible for some subgroup $\G$ of $\Aut_H(G)$. 
Then every ergodic faithful central $H$-modular operator space admits a unique $\G$-reasonable concrete operator space $G$-ification.
That is, if $(V, \pi, \Pi, \cK)$ and $(V^\prime, \pi^\prime, \Pi^\prime, \cK^\prime)$ are $\G$-reasonable $G$-ifications of $X$, then the map $\Pi(x) \pi(g)\mapsto \Pi^\prime(x) \pi^\prime(g)$ for $x\in X$ and $g\in G$, extends to a complete isometry $V\cong V^\prime$.
\begin{proof} 
First, we show that $X$ does admit at least one $\G$-reasonable concrete $G$-ification. 
Take any faithful central modular representation for the action of $H$ on $X$.   We can ensure that it is ergodic if $X$ is ergodic as follows. 
Let $(\theta^\prime, \rho^\prime, \cH^\prime)$ be our modular representation of $X$.
Let $p\in \cB(\cH^\prime)$ be the projection onto the subspace $\{\xi\in \cH^\prime : \theta^\prime(h) \xi = \xi\}$ of all $H$-invariant vectors in $\cH^\prime$. Then $\theta^\prime(h)p=p\theta^\prime(h)=p$ for every $h\in H$.
Since $X$ is ergodic, $\sum_{h\in H} hx=0$ for all $x\in X$. Thus, for every $\eta, \eta^\prime\in\cH^\prime$, and $x\in X$, we have
$$0=\langle \rho^\prime(\sum_{h\in H} hx)p \eta, \eta^\prime\rangle = 
\sum_{h\in H} \langle \rho^\prime(x)\theta(h)p \eta, \eta'\rangle= |H|\, \langle \rho^\prime(x)p \eta, \eta^\prime\rangle .$$
Thus, it follows $\rho^\prime(x)p = 0$ for all $x\in X$. Similarly, $p\rho^\prime(x) = 0$ for all $x\in X$.
The subspace $\cH:=(1-p)\cH^\prime$ is invariant under both $\theta^\prime$ and $\rho^\prime$, hence we get a 
faithful central modular representation $(\theta, \rho, \cH)$ for the action of $H$ on $X$, where $\theta$ and $\rho$ are the restrictions of $\theta^\prime$ and $\rho^\prime$ to $\cH$, respectively.

Let $\cC=\theta(H)^\prime$ be the commutant of $H$ in $\cB(\cH)$. Since $H$ is abelian, the von Neumann algebra $\cC$ is $H$-invariant, and turns into an $H$-centric involutive algebra, containing $\rho(X)$.
The $H$-actions on $\cH$ and $\cC$ are both ergodic by construction if $X$ is ergodic.

We show that $\cC$ admits a $\Aut_H(G)$-reasonable concrete operator space $G$-ification, and therefore we get a $\Aut_H(G)$-reasonable operator space $G$-ification of $X$ by restriction of the matricial norms.

Define $\Pi = \fI\circ\fE_\cC$.  The maps $\fI$ and $\fE_\cC$ are completely isometric $*$-homomorphisms  
by Proposition~\ref{prop:conv-rep} and Lemma~\ref{lem:red-reg-*}, respectively, and so the $*$-representation $\Pi$ is completely isometric. 
The quadruple $(\fI(\cC_G), \pi_{\cH}, \Pi, \cH_G)$ is a $G$-ification of $\cC$.
Indeed condition (ii) in Definition~\ref{def:conc-Gif} is obvious, and for (i) we observe 
for $g\in G$, $T \in \cC$ and $\vf\in \cH_G$,
\begin{align*}
[\fI(\fE_\cC(T)) \big(\pi_{\cH}(g) (\vf)\big)](g') 
&= \frac{1}{|H|} \sum_{k\in G} \fE_\cC(T)(k) \big(\vf(g^{-1}k^{-1}g')\big)
\\&=\frac{1}{|H|} \sum_{k\in H} \fE_\cC(T)(k) \big(\vf(k^{-1}g^{-1}g')\big)
\\&= [\pi_{\cH}(g)\big(\fI(\fE_\cC(T)) (\vf)\big)](g').
\end{align*}

To verify condition (iii) in Definition~\ref{def:conc-Gif}, recall from Proposition \ref{prop:conv-rep} 
that $\fI$ is $G$-equivariant, that is, $\fI(\pi_\cC(g) \, \Psi) = \pi_\cH(g) \fI(\Psi)$ for $g\in G$, and $\Psi\in \cC_G$.  It then follows from Lemma~\ref{lem:pol-id} that
\[
\fI(\cC_G) = \fI( \oplus_{k=1}^m \,  g_k \, \fE_\cC(\cC)) = \oplus_{k=1}^m \,   \Pi(\cC) \, \pi_{\cH}(g_k) ,
\]
where $\{ g_i : i = 1, \cdots , m \}$ is a full set of mutually inequivalent representatives of $G/H$.
 By Lemma~\ref{lem:Aut-cov-rep}, the action $\Aut_H(G)\act \cC_G$ is inner, hence acts via complete isometries.
Thus, this $G$-ification of $\cC$ is $\Aut_H(G)$-reasonable, and hence so is its subspace corresponding to $X_G$.

Towards the uniqueness part of the statement, we first claim that the above reasonable norm on $X_G$ is independent of the choice of the modular representation. 
Indeed, in the setup of the above construction, we see that $P_H:=\frac{1}{|H|}\sum_{h\in H} \lambda_h\otimes \theta(h)\in \cB(\ell^2(G, \cH))$ is the projection onto the subspace $\cH_G$. Thus, we have a completely isometric embedding of $\fI(X_G)$ into $\cB(\ell^2(G, \cH))$ by composing with the projection $P_H$.
Now, given a bijective completely isometric $H$-intertwiner $\dT: X\to Y$ between $H$-modular operator spaces, if $\fI(X_G)\subset \cB(\ell^2(G, \cH))$ and $\fI(Y_G)\subset \cB(\ell^2(G, \cK))$ are constructed as the above, then the $G$-ification $\dT_G$ of the map $\dT$ coincides with $\id\otimes \dT : \fI(X_G)\to\fI(Y_G)$, hence completely isometric. This shows the claim (which also follows from the later result Corollary \ref{thm:Gif-map-isom}).
In the rest of the proof, we refer to the above reasonable operator space $G$-ification as the `canonical' $G$-ification.

Next, we show the canonical reasonable norm above is unique. 
 Let $(V,\pi, \Pi, \cK)$ be a concrete $\G$-reasonable $G$-ification of $X$.
Let $q\in \cB(\cK)$ be the projection onto the subspace of $H$-invariant vectors. Then as we saw in the first part of the proof, $\Pi(X)q=q\Pi(X)=0$. Since $H$ is normal in $G$, $q$ commutes with $\pi(g)$ for all $g\in G$, and it follows $Vq=qV=0$. Thus, by reducing to the subspace $(1-q)\cK$, we may assume that the concrete $\G$-reasonable $G$-ification $(V,\pi, \Pi, \cK)$ is such that the $H$-actions on $\cK$ and on the commutant $\tilde\cA$ of $\pi(H)$ in $\cB(\cK)$ are both ergodic.

Thus, we may equip $\tilde\cA$ with its canonical $\Aut_H(G)$-reasonable $G$-ification operator space structure, which restricts to the canonical reasonable $G$-ification operator space structure on $X_G\cong \Pi(X)_G$.

Let $\cA$ be the commutant of $\pi(G)$ in $\cB(\cK)$. Then $\cA$ is an $H$-invariant $C^*$-subalgebra of $\tilde\cA$ containing $\Pi(X)$. Consequently, $\cA_G$ is a $C^*$-subalgebra of $\tilde\cA_G$ containing $\Pi(X)_G$.

Define $\mathfrak{m}: \cA_G\to \cB(\cK)$ by $\mathfrak{m}(\Psi) = \frac{1}{|H|} \,  \sum_{g\in G} \Psi(g) \pi(g)$. We claim that $\mathfrak{m}$ is a $*$-homomorphism, hence a contraction. To see the claim, let $\Phi, \Psi\in \cA_G$, then
\begin{align*}
\mathfrak{m}(\Phi*\Psi) 
&
= \frac{1}{|H|^2} \, \sum_{g\in G} \sum_{k\in G} \Phi(k)\Psi(k^{-1}g) \pi(g)
= \frac{1}{|H|^2} \, \sum_{k\in G} \Phi(k) \sum_{g\in G} \Psi(g) \pi(k)\pi(g)
\\&= \frac{1}{|H|^2} \, \sum_{k\in G} \Phi(k) \pi(k) \sum_{g\in G} \Psi(g) \pi(g)
= \mathfrak{m}(\Phi) \mathfrak{m}(\Psi) ,
\end{align*}
and 
\begin{align*}
|H| \,  \mathfrak{m}(\Psi^\ast) &= \sum_{g\in G} \Psi^\ast(g) \pi(g) = \sum_{g\in G} \Psi(g^{-1})^\ast \pi(g) = \sum_{g\in G} \Psi(g)^\ast \pi(g^{-1}) 
\\&= \sum_{g\in G} \Psi(g)^\ast \pi(g)^\ast = \left(\sum_{g\in G} \Psi(g) \pi(g)\right)^\ast = |H| \, \mathfrak{m}(\Psi)^\ast .
\end{align*}

Now, consider the *-homomorphism $\bigoplus_{\sigma\in\Aut_H(G)} \mathfrak{m}\sigma: \cA_G\to \bigoplus_{\sigma\in\Aut_H(G)}\cB(\cK)$. We show it is injective.

Indeed, 
for $\Psi \in \cA_G$ assume  that $\mathfrak{m}(\Psi\sigma) = 0$ for every $\sigma\in\Aut_H(G)$. Then $$(\mathfrak{m}\sigma)(\Psi) = \mathfrak{m}(\Psi \circ \sigma^{-1}) = 
\frac{1}{|H|} \, \sum_{g\in G} \Psi(\sigma^{-1}(g)) \pi(g) = \frac{1}{|H|} \,  \sum_{g\in G} \%zz
 \Psi(g) \pi(\sigma(g)) = 0 .$$
By Lemma~\ref{lem:sumzero}, $\sum_{\sigma\in \Aut_H(G)} \pi(\sigma(g)) = 0$ for every $g\in G\setminus H$. 
Thus,
\[
\sum_{\sigma\in \Aut_H(G)}\sum_{g\in G\setminus H} \Psi(g) \pi(\sigma(g)) = 0,
\]
 which implies by the above that
\begin{align*}
0 &= \sum_{\sigma\in \Aut_H(G)}\sum_{h\in H} \Psi(h) \pi(\sigma(h))
= \sum_{\sigma\in \Aut_H(G)}\sum_{h\in H} \Psi(h) \pi(h) 
\\&= |\Aut_H(G)| \sum_{h\in H} \pi(h^{-1}) \Psi(e) \pi(h)
= |\Aut_H(G)| \sum_{h\in H} 
\Psi(e) ,
\end{align*} 
which implies $\Psi(e) = 0$.
Let $k\in G$, and let $\Phi =\pi(k)(\Psi)$. Then for every $\sigma\in\Aut_H(G)$,
\begin{align*}
|H| \, \mathfrak{m}(\Phi\sigma) 
&= \sum_{g\in G} \Phi(g) \pi(\sigma(g)) 
= \sum_{g\in G} \Psi(k^{-1}g) \pi(\sigma(g)) 
\\&= \sum_{g\in G} \Psi(g) \pi(\sigma(kg)) 
= \sum_{g\in G} \Psi(g) \pi(\sigma(k))\pi(\sigma(g)) 
\\&= \pi(\sigma(k))\sum_{g\in G} \Psi(g) \pi(\sigma(g)) 
= \pi(\sigma(k))\mathfrak{m}(\Psi\sigma) 
= 0 ,
\end{align*} 
which implies by the above argument that $\Psi(k^{-1}) = \Phi(e) = 0$. Hence, $\Psi=0$, and the claim follows.

The claim in particular implies that $\| \Psi \| = \sup_{\sigma\in\G} \|\mathfrak{m}(\Psi\sigma)\|$ for every $\Psi\in \cA_G$. 
Since $(\pi, \Pi, \cK)$ is a $\G$-reasonable $G$-ification of $X$, we have $\|\mathfrak{m}(\Psi\sigma)\| = \|\mathfrak{m}(\Psi)\|$ for every $\sigma\in \G\subset \cA_G$, and the same holds at each matrix level.
Thus, it follows that $\mathfrak{m}$ is completely isometric on $\Pi(X)_G$. 

Obviously, $\mathfrak{m}$ maps $\Pi(X)_G$ onto $\sum_{g \in G} \Pi(X)\pi(g)$. Hence, the operator space structure on $X_G$ inherited by the representation $(\pi, \Pi, \cK)$ coincides with the canonical $G$-ification.
\end{proof}
\end{thm}

\begin{remark}
We note that in the existence part of the above proof, we did not use the admissibility condition, it was used only in the uniqueness part. In particular, we showed that:
\emph{ If $H\le Z(G)$, then every central $H$-modular operator space admits an $\Aut_H(G)$-reasonable operator space $G$-ification. Furthermore, if $X$ is $H$-ergodic 
then the concrete $G$-ification can be chosen so that $H$ acts ergodically on the Hilbert space of the representation. }
Moreover we remark that the admissibility was only used in the uniqueness to allow access
to Lemma \ref{lem:sumzero}.  In fact for the uniqueness part of the above proof
one could replace admissibility by the weaker condition 
$\sum_{\sigma\in \G} \pi_X(\sigma(g)) = 0$ for $g\in G\setminus H$ established in
that lemma.   

 We do not know to what extent one may relax the last condition 
 and still get uniqueness.   For any abelian $G$ we are able to show that 
the canonical $G$-ification $X_G$ has the largest norm
(and matrix norms) among all $G$-ifications.  This should help in future `uniqueness theorems'.
\end{remark}

\begin{cor}
Assume $H\le Z(G)$, $G$ is finite, and that the pair $(G, H)$ is $\G$-admissible for some subgroup $\G$ of $\Aut_H(G)$.
If $X$ is an ergodic faithful central $H$-modular operator space, then every $\G$-reasonable concrete operator space $G$-ification of  $X$ is $\Aut_H(G)$-reasonable.
\begin{proof}
As remarked above, every ergodic faithful central $H$-modular operator space $X$ has an $\Aut_H(G)$-reasonable concrete $G$-ification, which is completely isometric to any $\G$-reasonable concrete $G$-ification of $X$. Thus the claim follows.
\end{proof}
\end{cor}

\begin{example}
Let $X$ be a real operator space and $X_\bH$ its quaternification (see Example~\ref{ex:quat}). If an operator space norm on $X_\bH$ (which extends the operator space norm of $X$) is invariant under the inner automorphisms of $Q_8$, then by the Corollary it is automatically invariant under every automorphism of $Q_8$ (recall $|{\rm Inn}(Q_8)|=4$ and $|\Aut(Q_8)|=24$).
\end{example}


\section{Abstract characterizations} \label{absc} 

\subsection{Characterizations of modular operator spaces} \label{chmod}

If $G$ is a group acting on an abstract operator space $X$, we wish to treat  group action constructions 
analogously to the theory of operator modules (see e.g.\ Chapters 3 and 4 in \cite{BLM}).  
Viewing the left and right $G$-actions as actions of $C^*(G)$, then (real or complex) $G$-modular operator spaces are nothing but (nondegenerate)  
$C^*(G)$-operator bimodules in the sense of Christensen-Effros-Sinclair.   These are also called $h$-bimodules  in \cite[Chapters 3 and 4]{BLM}, where one may find their theory developed (in the complex case; the real case works out almost identically
  as noted in \cite{BReal}).
   We explain some aspects of the latter theory in the setting of $G$-modular spaces. 

\begin{thm} \label{ces}  \begin{itemize} 
\item[(1)] The left and right $G$-actions on a (real or complex) $G$-modular operator space $X$ commute automatically, indeed there exists
a completely isometric  representation $\Pi : X \to B(\cH)$ and a unitary representation $\pi$ 
of 
$G$  on $\cH$ such that $$\Pi( (h x) k) = \Pi( h (xk) ) =  \pi(h) \Pi(x) \pi(k) , \qquad h, k\in G, x\in X .$$   If in addition
$X$ is
a central $G$-modular operator space then $\pi(G)$  commutes with $\Pi(X)$, and if also the $G$-action is faithful then $\pi$ is one-to-one.
\item[(2)] Viewing left and right $G$-actions as actions of $C^*(G)$, there is a bijective correspondence between
 $G$-modular operator spaces and nondegenerate 
$C^*(G)$-operator bimodules in the sense of Christensen-Effros-Sinclair.   Similarly for left or right $G$-modular operator spaces.
  \end{itemize} 
\end{thm}

\begin{proof}   If $X$ is left $G$-modular  there exists a completely isometric representation $\Psi : X \to B(\cH)$ and a unitary representation  
$\pi$ of $G$ on $\cH$ such that $\pi(g) \Psi(x) = \Psi(g x)$ for $g \in G, x \in X$.   Extending $\pi$ to a representation of $C^*(G)$ on $\cH$ 
we see that $\Psi(X)$ and $X$ are  nondegenerate left $C^*(G)$-operator modules.   Conversely any nondegenerate 
left $C^*(G)$-operator module $X$ has a `CES representation' $\Psi : X \to B(\cH)$  with the desired property showing that $X$ is left $G$-modular (see 
\cite[Theorem 2.4]{BReal} 
for the real version of this).   It follows from \cite[Theorem 4.6.7]{BLM} 
(whose  real version is similar; see \cite{BReal}) 
that $X$ is a left $C^*(G)$-operator module and a right $C^*(G)$-operator module if and only if it is a $C^*(G)$-operator bimodule. 
In particular, the left and right $G$-actions commute. 
Similarly to the `left case' we get a representation as in (1) (again see 
\cite[Theorem 2.4]{BReal} 
for the real version of the operator bimodule representation).      

 If in addition 
$X$ is central then $\Pi( h x) = \Pi( xh ) =  \pi(h) \Pi(x) = \Pi(x) \pi(h)$.  If the action is faithful and $\pi(h) = I$ then $\Pi( h x) = \Pi( x)$ and 
$hx = x$ for all $x \in X$, so that $h = 1$. 
\end{proof}

\begin{remark}
One may characterize the left $G$-modular  (resp.\  central $H$-modular) actions on an operator space $X$ in terms of unitaries in 
the operator space left multiplier $C^*$-algebra $\cA_\ell(X)$ (resp.\  centralizer algebra 
$Z(X) = \cA_\ell(X) \cap \cA_r(X)$) and their theory 
see e.g.\ \cite[Proposition 4.5.8]{BLM}, \cite[Corollary 4.8]{BEZ}, and \cite{Sharma} and \cite[Section 4]{BReal} in the real case.   
The cited corollary gives the metric condition
 $$\big\| \big[ \begin{array}{cl} g x_{ij} \\ y_{ij} \end{array} \big] \big\| = \big\| \big[  \begin{array}{cl} x_{ij} \\ y_{ij} \end{array} \big] \big\| , \qquad g \in G, [x_{ij}], [y_{ij}] \in M_n(X).$$ 
characterizing  left $G$-modular actions, or equivalently when $G$ acts as unitary elements of the $C^*$-algebra $\cA_\ell(X)$.    Right $G$-modular actions are 
characterized by the condition $\| [ [g x_{ij} ] \; [ y_{ij} ] ] \| = \| [ [x_{ij} ] \; [ y_{ij} ] ] \|$, and  
central $H$-modular actions are characterized by the one action being both left and right modular.

As in these references it is sometimes useful to view such structure in terms of the `Shilov representation' (that is the ternary or $C^*$-envelope), or injective envelope) of $X$ or of $X_G$.
For brevity we shall 
 avoid this perspective here.  
\end{remark}

In the following 
we combine the two cases of `$H$-centric algebras' and `left-involutive algebras' considered in Sections 2.1.2 and 2.1.3.

\begin{prop}\label{isinv} 
If $A$ is a real or complex $C^*$-algebra and $H$ is an abelian group acting faithfully on the left of $A$ by invertible linear maps.
 The following  are equivalent: 
 \begin{itemize} 
\item [(1)] 
 $A$ is an $H$-centric algebra and the left action of $H$ is left-involutive.
\item [(2)]  The action of $H$ on $A$ is  by left multiplication by unitaries  in the center ${\mathcal Z}(M(A))$ of $M(A)$. 
\item [(3)]  $A$ is a central $H$-modular operator space. \end{itemize}
\end{prop} 

\begin{proof}    (1) $\Rightarrow$ (2) \ We have $$(h \cdot x)^* y = (h^{-1} \cdot x^*) y =  x^* (h^{-1}   \cdot y), \qquad x, y \in A, h \in H .$$ 
Thus if $T_h x = h \cdot x$ then $T_h$ is adjointable with adjoint  $T_{h^{-1}} = (T_h)^{-1}$. 
It follows that $T_h$ is unitary as an adjointable map.  For any $C^*$-algebra viewed as a right $C^*$-module
the adjointable maps correspond to left multiplication by elements in $M(A)$.  
So $H$ corresponds to  a subgroup of the unitaries in $M(A)$. 
Since $(hx)y = x(hy) = (xh) y$ for $x, y \in A$ we see that 
$h \in {\mathcal Z}(M(A))$.

(2) $\Rightarrow$ (1) \  We have  $(h x) y = x (h y)$ for $x, y \in A, h \in H$.
So $A$ is $H$-centric.   Also, 
$(h x)^* = x^* h^{-1} = h^{-1} x^*$. 

(2) $\Rightarrow$ (3) \   This implication is obvious.

(3) $\Rightarrow$ (2) \   We saw in Theorem \ref{ces} that central $H$-modular operator spaces correspond to operator 
$C^*(H)$-bimodules with $hx = xh$ for $x \in A, h \in H$.  By the theory of operator 
bimodules this corresponds to a $*$-homomorphism $\pi : C^*(H) \to \cA_l(A) \cap \cA_r(A) = Z(A) = {\mathcal Z}(M(A))$.
This uses facts from 4.6.6 in \cite{BLM} and e.g.\ \cite[Section 4]{BReal} in the real case (see e.g.\ Example (1) before Theorem 4.4 there).
Thus we have (2) since $\pi(h)$ is  unitary, being a contraction with contractive inverse. 
\end{proof}

If these equivalent conditions hold then by (2) $A$ has a $*$-representation in which it commutes with $H$.  Thus $A$ is a
$C^*$-subalgebra and $H$-submodule of $\cC = H'$, so that $A_G$ is a $C^*$-subalgebra of $\cC_G$.


\begin{remark}
A similar proof shows that if $A$ is a unital operator algebra then $A$ is  a central $H$-modular operator space
 if and only if 
the  left action of $H$ on $A$ is by left multiplication  by unitaries  in the (diagonal of the) center of $A$; and if and only if 
$A$ is a $H$-modular operator space and  an $H$-centric algebra.  Since $A \subset H'$ we then have as above that $A_G$ is an operator subalgebra
of the $C^*$-algebra $\cC_G$. 
Similar results hold for approximately unital operator algebras.  If $\cA$ is a von Neumann algebra then 
so is $\cA_G$ for finite $G$ (this can be proved with the help of Theorem \ref{iso}).
\end{remark}

Similarly, the $G$-ification of an operator system is an operator system (this may be seen easily from its matrix form in Section \ref{marp}). 
Similarly, the $G$-ification of a unital operator space  (resp.\ operator module over $A$) 
can be shown by similar arguments to be a unital  operator space (resp.\ operator module over $A_G$).  We omit the details since we shall not need these here. 

\subsection{Abstract  $G$-ification}

\begin{defn} \label{def:abs-Gif}
Let $X$ be a  faithful central $H$-modular operator space.
By an {\em abstract operator space $G$-ification of $X$} we mean a  pair $(V, \Phi)$
where $V$ is a $G$-modular  operator space  with
 $hv = vh$ for $v \in V, h \in H$, 
and $\Phi : X \to V$ is a complete isometry and $H$-intertwiner into $\{ v \in V : g v = v g \; \textrm{for all} \, g \in G \}$ such that 
$V = \bigoplus_{gH \,\in\, G/H} \, \Phi(X) \, g$. \end{defn}

We identify two operator space $G$-ifications  $(V, \Phi)$ and $(W, \Psi)$ of $X$ if there exists a surjective complete isometry 
$j : V \to W$ such that 
 $j (\Phi(x) \, g) = \Psi(x) g$ for $x \in X$ and $g \in G$.
Notice that this implies that $j \circ \Phi = \Psi$ and $j(v g) = j(v) g$ and $j(g v) = g j(v)$ for $v \in V$ and $g \in G$.

Every concrete operator space $G$-ification $(V, \pi, \Pi, \cK)$ of $X$ may be viewed as an abstract operator space $G$-ification of $X$. Indeed $V=\sum_{gH \,\in\, G/H}  \Pi(X)\pi(g)$ clearly is a $G$-modular  operator space; or alternatively an operator $C^*(G)$-bimodule. Since $\pi(H)$ commutes with $\pi(G)$ and $\Pi(X)$ it is easy to see that $hv = vh$ for $v \in V, h \in H$.  
The other conditions of being an  abstract operator space $G$-ification are obvious. The next result shows that every abstract operator space $G$-ification of $X$
may be identified in the sense above with  an appropriate concrete operator space $G$-ification.   Thus for the most part we may treat concrete and abstract operator space $G$-ifications interchangeably.

\begin{thm} \label{absco} 
Let $G$ and $H$ be as above, with $H\le Z(G)$,  and let $(V, \Phi)$ be an abstract operator space $G$-ification of $X$.
Then there exists a unitary representation $\pi$ of $G$ on a 
Hilbert space $\cK$, and a completely isometric representation $\Psi: V \to \cB(\cK)$  such that $(\pi, \Psi \circ \Phi, \cK)$ yields a concrete operator space $G$-ification of $X$ such that $(\sum_{g\in G} \, \Psi(\Phi(X)) \pi(g), \Psi \circ \Phi)$ is identifiable with $(V, \Phi)$ in the sense above.
\end{thm}

\begin{proof}    By Theorem \ref{ces} there exists
a completely isometric  representation $\Psi : V \to B(\cH)$ and a unitary representation $\pi$ 
of $G$ on $\cH$ such that $$\Psi( (h x) k) = \Psi( h (xk) ) =  \pi(h) \Psi(x) \pi(k) , \qquad h, k\in G, x\in V.$$   
Let $\Pi = \Psi \circ \Phi : X \to B(\cH)$,  which is a complete isometric representation of $X$.  
Then $$\Pi(x) \pi(g) = \Psi (\Phi (x))  \pi(g) = \Psi (\Phi (x) g) = \Psi (g \Phi (x)), \qquad x \in X, g \in G,$$ which
by similar considerations equals  $\pi(g) \Pi(x)$.  Also, 
$$\Pi(hx) =  \Psi(\Phi(hx)) = \pi(h)\circ \Pi(x), \qquad x \in X, h \in H.$$   Finally, $$
\Psi(V)=\bigoplus_{gH\,\in\, G/H} \, \Psi(\Phi(X) g) = \bigoplus_{gH\,\in\, G/H} \, \Pi(X) \pi(g) .$$  Thus  $(\Psi(V) , \cK, \Psi \circ \Phi, \pi)$
is a concrete operator space $G$-ification of $X$. 
\end{proof}

In the above setup, the map $V\to X_G$ defined by $\sum_{Hg\,\in\, H\backslash G}\Pi(x_g)\pi(g)\mapsto (g\mapsto x_g)$ is a linear (algebraic) isomorphism. In particular, this allows us to define the
 $\Aut_H(G)$-action on $V$.   We may then  again say that an abstract $G$-ification $V$ of $X$ is \emph{reasonable} if 
the action of $\Aut_H(G)$ on $V$ is by complete isometries.

 \begin{thm} \label{thm:Gif-unqabs}
Assume $H\le Z(G)$, $G$ is finite, and that the pair $(G, H)$ is $\G$-admissible for some subgroup $\G\le\Aut_H(G)$. 
Then every ergodic faithful central $H$-modular operator space admits a unique ($\G$-) reasonable abstract operator space $G$-ification. \end{thm}

Indeed by Theorem  \ref{absco} any abstract operator space $G$-ification is identifiable with a concrete  $G$-ification.

\section{Applications, further properties and equivalent descriptions} \label{applns} 

\subsection{A matrix representation of the $G$-ification} \label{marp}

An essential feature of the theory of the complexification $X_c$ of an operator space $X$ is, as we said in the introduction, that $X_c$ is representable completely isometrically as 
the subspace of $M_2(X)$ consisting of the matrices in equation (\ref{ofr}).
    Similarly for the quaternification.   We now establish the same fact for the $G$-ification. 
    Just one of the very many advantages of this is that the previously elusive norm on the $G$-ification becomes a well understood 
and tractable norm, involving only  $X$, of such matrices.  Indeed 
this norm is a simple artifact of the operator space structure (i.e.\ given matrix norms) of $X$.   Several explicit examples are displayed in  the last section of our paper.

We assume below that $G$ is finite and $H\le Z(G)$.
Suppose that $\{ g_i : i = 1, \cdots , m \}$ are a full set of mutually inequivalent representatives of $G/H$, where $m = |G/H|$.   We take $g_1 = 1$.   For each pair $i, j$ there exists $r_{ij} \in \{ 1, \cdots, m \}$ with  $g_i g_j^{-1} \in g_{r_{ij}} H$.  
For each $i,j$ define fixed elements $h_{ij} \in H$ so that $g_i g_j^{-1} = g_{r_{ij}} h_{ij}$.   Thus we have $r_{ii} = h_{ii} = h_{i1} = 1$, and $r_{i1} = i$, for all $i$.
Note that for any $i$, the sequence $(r_{ij})_{j=1}^m$ is a permutation  of $\{ 1, \cdots , m \}$.   Similarly, 
for any $j$, the sequence $(r_{ij})_{i=1}^m$ is a permutation of $\{ 1, \cdots , m \}$.      For any $H$-modular operator space $X$ we define 
{\em the matrix form} $G(X)$ of the $G$-ification to be the subspace 
$$G(X) \, = \, \{ \, [h_{ij}^{-1} \, x_{r_{ij}}] \, \in \, M_m(X) \, : \, x_1 , \cdots , x_m \in X \, \} .$$
 This is an operator space.  Define $V_k = [\delta^k_{i,j} h_{ij}^{-1}]
\in M_m(H)$ for $k = 1, \cdots, m$,  where  $h_{ij}$ are as above, and $\delta^k_{i,j} = 1$ if $r_{ij} = k$, and $\delta^k_{i,j} = 
0$ otherwise.     These turn out to be unitary (multipliers), and $G(X) = \oplus_k \, V_k \, (I_m \otimes X)$.

\begin{lemma} \label{un}   Assume $H\le Z(G)$, and let $\cH$ be a Hilbert space with $H$ acting faithfully
as unitaries on $\cH$.
Consider the canonical unitary $W : \cH_G \cong \cH^{(m)}$ taking $\varphi$ to 
$\sqrt{m} \, [\varphi(g_i)]$.
Then identifying $V_k \in M_m(H)$ above with an operator on $\cH^{(m)}$ via the 
inclusions $M_m(H) \subset  M_m(B(\cH)) \cong B(\cH^{(m)})$ we have 
 $V_k = W \pi_{\cH}(g_k) W^*$.  Also, $h \otimes I_m =  W \pi_{\cH}(h) W^*$ for $h \in H$.    
Moreover $G(X) = \oplus_k \, (I_m \otimes X) \, V_k = \oplus_k \, V_k \, (I_m \otimes X)$
for any central $H$-modular subspace $X \subset B(\cH)$. 
\end{lemma} 

\begin{proof}    We have $(\pi(g_k) (\varphi))(g_i) = \varphi (g_k^{-1} g_i)$.   On the other hand, define $\delta^k_{i,j} = 1$ if $r_{ij} = k$, and $\delta^k_{i,j} = 
0$ otherwise.  Then for $\varphi \in \cH_G$ we have 
$$V_k ([\varphi(g_i)]) = [h_{ij}^{-1} \delta^k_{i,j}] [\varphi(g_i)] = [\sum_j \, h_{ij}^{-1} \delta^k_{i,j} \varphi(g_j) ] = [\sum_j \, \delta^k_{i,j} \varphi( h_{ij} g_j) ].$$ 
Now $g_i g_j^{-1} = g_{r_{ij}} h_{ij}$, so that $h_{ij} g_j = g_{r_{ij}}^{-1} g_i$.   Thus the expression at the end of the last centered equation is $[\varphi(g_k^{-1} g_i) ]$.
That is, $V_k ([\varphi(g_i)]) = [(\pi(g_k) (\varphi))(g_i)]$, or $V_k W  \varphi = W \pi(g_k)  \varphi$.  
Similarly $W \pi(h) \varphi = \sqrt{m} \, [h \, \varphi(g_i)] = (h \otimes I_m) W  \varphi,$ which yields $h \otimes I_m =  W \pi_{\cH}(h) W^*$.

 Finally, any $x \in G(X)$ is 
 of form 
$$[h_{ij}^{-1} \, x_{r_{ij}}] =  \sum_{k=1}^m \, (I_m \otimes x_k) V_k =  \sum_{k=1}^m \, V_k \, (I_m \otimes x_k)$$ 
as asserted.  The direct sum assertion is clear from the `permutation matrix form' of the $V_k$, or follows
from the earlier matching fact for $X_G$, with $\pi(g_k)$ in place of $V_k$. \end{proof}

As before, if $H$ acts faithfully  on $\cH$ by unitaries, let $\cC$ be the commutant of $H$ in $\cB(\cH)$.   By the 
lines after Proposition \ref{isinv}  any central $H$-modular  $C^*$-algebra $\cA$   
 also satisfies the conditions of the next result
for some $\cH$ and $\cC$, and $\cA_G$ 
may be viewed as a $C^*$-subalgebra of $\cC_G$. 

\begin{thm} \label{iso} Assume $H\le Z(G)$, and that $\cA$ is a
 $C^*$-subalgebra and a $H$-submodule of $\cC$.
Let $X$ be a central $H$-modular operator space. 
 Then $G(\cA)$ is a $C^*$-algebra  $*$-isomorphic to $\cA_G$, and 
 $G(X) \cong X_G$ completely isometrically and as abstract $G$-ifications of $X$. \end{thm}

\begin{proof}    Define $\theta : \cC_G \to G(\cC)$ by $\theta(\varphi) = 
[\varphi(g_i g_j^{-1})]$.   
It is an exercise that $\theta$ is a $*$-homomorphism, as is its restriction to $\cA_G$. 
It is one-to-one since if $\varphi(g_i ) = 0$ for all $i$ then clearly $\varphi(g) = 0$ for all $g \in G$.  
We have $$[\varphi(g_i g_j^{-1})] = [\varphi(g_{r_{ij}} \, h_{ij})] = [h_{ij}^{-1} \, \varphi(g_{r_{ij}})] =  \sum_{k=1}^m \, (I_m \otimes \varphi(g_k)) V_k .$$ 
It follows that  $\theta(\cA_G) \subset G(\cA)$.
Given $x_1 , \cdots , x_m \in \cA$  with $[h_{ij}^{-1} \, x_{r_{ij}}] \in G(\cA)$ define $\varphi_0 \in \cA_G$ by  $\varphi_0(g_i h) = h^{-1} x_i$  for $h \in H$.
We have $$\varphi_0(g_i g_j^{-1})= \varphi(g_{r_{ij}} \, h_{ij}) = h_{ij}^{-1} \, \varphi(g_{r_{ij}}) = h_{ij}^{-1} \, x_{r_{ij}} .$$
So $\theta(\cA_G) = G(\cA)$. 

If $X$ is a central $H$-modular operator space then we may assume as in the proof of Theorem \ref{thm:Gif-unq}
  that $X$ is represented completely isometrically  as an $H$-submodule of  $\cC$.
 The restriction  of $\theta$ to $X_G$ is completely isometric, and the argument in the last few lines of the last paragraph shows that $\theta(X_G) = G(X)$.  
 If $x_k$ and $\varphi_0$ are as above then since $\theta( \sum_{k=1}^m \, \Pi(x_k) \, \pi(g_k)) =  \theta(\varphi_0) = \sum_{k=1}^m \, (I_m \otimes x_k) \, V_k$, the identification
 is as abstract $G$-ifications. 
\end{proof}

In many of the following proofs
we view the $G$-ification $X_G$ as the subspace
$G(X) \subset M_m(X)$ above where $m = |G/H |$.

\begin{cor}\label{thm:Gif-map-isom}
Assume $H\le Z(G)$. 
Let $X$ and $Y$ be central $H$-modular  operator spaces. An $H$-intertwiner  $T:X\to Y$ is completely isometric (resp.\ completely contractive) if and only if 
 $T_G$ is completely isometric (resp.\ completely contractive).
\begin{proof}
Clearly $T_G$ may be identified with $T \otimes I_m$ restricted to $G(X)$. 
\end{proof}
\end{cor}

\begin{lemma}   \label{compl} Assume $H\le Z(G)$. View the $G$-ification $X_G$ of a central $H$-modular space $X$ as the subspace
$G(X) \subset M_m(X)$ as 
above, where $m = |G/H |$.
There is a completely contractive projection from $M_m(X)$ onto $G(X)$.  
\end{lemma} 

\begin{proof}  Define $\Phi([x_{ij}]) = (\frac{1}{m} \sum_k \,  x_{kk}) \, I_m \in G(X)$; this fixes 
$I_m \otimes X$. Then 
consider the average  $$E(a) = 
 \frac{1}{m} \sum_k \,  \Phi(V_k^* a)  \, V_k \in G(X).$$   
Here $a \in M_m(X) \subset M_m(B(H)) \cong B(H_G)$. 
 Since $V_k \in M_n(H)$, the expression $V_k^* a$ is in $M_m(X)$ 
and $\Phi(V_k^* a) \in X \otimes I_m$. 
 Then $E$ is completely contractive and  fixes $G(X) = \sum_k \, (I_m \otimes  X) \,V_k$.  Indeed the latter `fixing' 
 is the `polarization identity' (see the remarks after Lemma~\ref{lem:pol-id}), since $\Phi$ extends the canonical expectation of $G(X)$ onto $I_m \otimes  X$.    \end{proof}

\subsection{Subspaces, quotients and subgroups}

We assume below that $G$ is finite and $H\le Z(G)$.  Several known properties  of the complexification have appropriate versions for the $G$-ification, but given the length
of our paper we just give a couple. 

\begin{prop}
We have $X_G / Y_G \cong (X/Y)_G$ completely isometrically if $Y$  is a closed $H$-invariant subspace of a central $H$-modular operator space $X$.  
\begin{proof} 
Consider the canonical completely contractive map  $$q_G : X_G \to (X/Y)_G \subset 
M_m(X/Y) = M_m(X)/M_m(Y).$$  Since $q_m : M_m(X) \to M_m(X/Y)$ is a quotient map we may lift an element $z = q_G(x)$ of norm $< 1$ in $(X/Y)_G$, to $x \in X_G$ plus $y = [y_{ij}] \in M_m(Y)$ 
with $\| x + y \| < 1$.  Let $E$ be the projection in Lemma \ref{compl}. 
Then $E(y) \in G(Y)$ and $\| x + E(y) \| = \| E(x + y) \| < 1$ has quotient $q_G(x + E(y)  ) = q_G(x) = z$.   A similar argument works at the matrix level.  
\end{proof} 
\end{prop}

\begin{prop}  
Assume that $(G, H)$ is $\G$-admissible for some subgroup $\G\le\Aut_H(G)$.
Let $X$ be a central $H$-modular operator space. Then the restriction of $X_G$ to $Z(G)$ is completely isometrically isomorphic to $X_{Z(G)}$, that is,
\[
X_G|_{Z(G)} := \{\vf|_{Z(G)} : \vf\in X_G\} = X_{Z(G)} .
\]
\begin{proof}
To see this let $\Lambda=\{\sigma|_{Z(G)} : \sigma\in \G\}\le \Aut_H(Z(G))$.  Then the pair $(Z(G), H)$ is $\Lambda$-admissible.
It is obvious from Definition~\ref{def:conc-Gif} that the canonical representation of the $G$-ification $X_G$, when restricted to $X_G|_{Z(G)}$ is a concrete $Z(G)$-ification of $X$. This is also $\Lambda$-reasonable by Lemma~\ref{lem:Aut-cov-rep}. Hence the result follows from Theorem~\ref{thm:Gif-unq}.
\end{proof}
\end{prop}

\begin{thm}\label{thm:Gif-stages}
Assume $H_1\le H_2\le Z(G)$ and that all the pairs $(G, H_2)$, $(G, H_1)$, and $(H_2, H_1)$ are $\G$-admissible for some subgroup $\G\le\Aut_H(G)$.
Let $X$ be a central $H_1$-modular operator space that is $H_1$-ergodic. Then 
the $G$-ification of the $H_2$-ification $X_{H_2}$ of $X$ coincides with the $G$-ification of $X$, that is,
\[
X_G = (X_{H_2})_G .
\]
\begin{proof}
This follows immediately from Theorem~\ref{thm:Gif-unq}.  Indeed if $\{ k_l \}$ is a full set of coset representatives for $H_2/H_1$, and if $\{ b_r \}$ is a full set of coset representatives for $G/H_2$ then $\{ k_l b_r \}$ are a full set of mutually inequivalent coset representatives in
$G/H_1$.   We have that $(X_{H_2})_G$  is an operator bimodule over $C^*(G)$ and it has a 
decomposition $\sum_{r} \,  \, X_{H_2} \, b_r = \sum_{r,l} \, X \, k_l b_r$.  So by uniqueness it is the $G$-ification of $X$. 
\end{proof}
\end{thm}


\section{The extended module action}\label{sec:mod}
We now turn our attention to the algebras $\cG^\ell$ and $\cG^r$, and the norm structure they claim via their module actions on the $G$-ification.
We assume below that $G$ is finite and $H\le Z(G)$, and that $X$ is a faithful  central $H$-modular operator space.

We may by the discussion in the Remark after Theorem \ref{ces}  
regard $H$ as a subgroup of the unitaries in the centralizer $Z(X)$ of $X$, and let $D$ be the (commutative) $C^*$-subalgebra generated by $H$ in $Z(X)$, which 
equals span$(H)$ there. 
We call the $C^*$-algebra $D_G$ the {\em acting $C^*$-algebra} of the $G$-ification. 
  Note that $D_G$ is central $H$-modular, indeed $H \cong\fE_D(H) = H 1_G$ may  be viewed as a central subgroup of the unitaries in $D_G$. 
  Here $1_G$ is the identity of $D_G$. 
  This map on $H$ may be extended to a surjective $*$-homomorphism $\alpha :  C^*(G) \to D_G$.  Indeed the map $g \mapsto g 1_G$ is a 
group homomorphism into the unitaries in $D_G$.  (We write $g 1_G$ for $\pi_D(g) \, 1_G$.) So there is a $*$-homomorphism $\alpha : C^*(G) \to D_G$.    The image of $G$ 
under this homomorphism generates $D_G$.  Indeed  $D_G$ is spanned by 
terms $g \fE_D(d)$ for $g \in G, d \in D$, and $D$ is spanned by $H \subset G$, and $\alpha(g h) = gh 1_G = g \fE_D(h)$, for $h \in H$. 
Thus  $\alpha$  is surjective. 

\begin{example} \label{61}
 If $H = \{ 1, -1 \}$ acting ergodically on $X$,
 then as a subgroup of the unitaries in $Z(X)$ it is easy to see that $H = \{ \pm I \}$.
Thus $D = \bF 1$, where $\bF$ is $\bR$ or $\bC$, since this is the $C^*$-algebra generated by $I$ and $-I$.  So  the acting
 $C^*$-algebra  $D_G$ is $\bF_G$.   From the matrix form in Section \ref{marp} it is easy to see in this case that $X_G \cong X \otimes_{\rm min} D_G$, the spatial (minimal) tensor product.
 
If $G = \bZ_4$ and $\bF = \bR$ then $D_G = \bR_G \cong \bC$,  so that the $D_G$-module $X_G$ is 
an operator  $\bC$-module; that is, 
a complex operator space. Similarly if $G = Q_8$ and  $D_G = \bR_G \cong \bH$,  then the $D_G$-module $X_G$ is 
an operator  $\bH$-bimodule,
and is a `quaternionic operator space' and the quaternification of $X$.   Note that $D_G \cong G(D) \subset M_4(\bR)$ is precisely the usual representation of the 
quaternions in  $M_4(\bR)$ (see \ref{ExQ}). \end{example}

\begin{thm}  Suppose that $G$ is finite and $H\le Z(G)$, and that $X$ is a faithful central $H$-modular operator space.  The subalgebra $\cG^\ell$ of ${\rm CB}(X_G)$ is completely isometrically isomorphic to $D_G$.   Indeed $\cG^\ell$ with its natural involution
(and inherited norm from
 ${\rm B}(X_G)$) 
is a $C^*$-algebra $*$-isomorphic to $D_G$.  Similarly $\cG^r$ with its natural involution
(and inherited norm and `reversed product' from 
${\rm B}(X_G)$) is a $C^*$-algebra $*$-isomorphic to $D_G$, so that  $\cG^r \cong \cG^\ell$ $*$-isomorphically 
and completely isometrically with respect to the inherited operator space 
structure from $CB(X_G)$. 
\end{thm}

\begin{proof}  We may  view $X_G$ as the matrix space $G(X) \subset M_m(X)$, 
and $D_G$ as the matrix space $G(D)$, as in  Section \ref{marp}.  
   Since $X$ is a left operator module over $Z(X)$ and its subalgebra $D$, we have that $M_m(X)$ is a 
 left operator module over $M_m(D)$ and over its subalgebra $G(D)$.    Now  $G(D) \, G(X) \subset G(X)$, 
as may be seen for example 
 by observing that $D_G$ acts on $X_G$ on the left by convolution by a slight variant of the calculation in Lemma \ref{lem:X_G-alg}. 
 Note that this action of $g 1_G \in D_G$ on $X_G$ is exactly $\pi_X(g)$ for $g \in G$. 
In particular $g_k 1_G$ becomes $V_k \in M_m(H)$, and $V_k G(X) \subset G(X)$.   Thus 
 $G(X)$ is a $G(D)$-submodule of $M_m(X)$.  Hence $G(X)$ is a  left operator $G(D)$-module.   
 
 By the theory of operator modules (see e.g.\ 4.6.6 in \cite{BLM}, with modifications from  e.g.\ \cite[Section 4]{BReal} in the real case)  
 there is a  $*$-homomorphism $\lambda : G(D) \to \cA_l(G(X))$ implementing the left module action, and the latter $C^*$-algebra is a subalgebra of $B(G(X))$ 
isometrically (resp.\ $CB(G(X))$ completely isometrically).  If $\lambda(\sum_k \, V_k \, (d_k \otimes I_m) ) = 0$ for $d_k \in D$ 
 then $\sum_k \, V_k \, (d_k \otimes I_m)  (X \otimes I_m) = 0$.  But this implies that $(d_k \otimes I_m)  (X \otimes I_m) = 0$ for all $k$, 
 so that $d_k = 0$.  So $\lambda$ is faithful, hence is completely isometric.   
 We claim that its range is $\cG^\ell$, the  span in $CB(X_G)$ of 
 $\pi_X(G)$.  Note  $\lambda(\alpha(g)) = \lambda(g 1_G)$, which as we said corresponds precisely to the action $\pi_X(g)$ on $X_G$. 
 Hence $\lambda(D_G) = \lambda(\alpha(C^*(G)))$ is  the linear span in $CB(X_G)$ of $\pi_X(G)$.   

 A similar proof works for the map $\rho : D_G \to \cA_r(X_G)$ implementing the right action of $D_G$ on $X_G$, and for $\cG^r$.  
This $\rho$ is a faithful (completely isometric) $*$-homomorphism into the $C^*$-algebra $\cA_r(X_G)$.   Now $\cA_r(X_G) \subset CB(X_G)$ 
completely isometrically, however
 $\rho$ is an antihomomorphism  as a map into $CB(X)$ with its usual product (as is usual for right actions).    
 It is easy to check that  $\rho(d 1_G)  = \pi_X^r(g^{-1})$ for $g \in G$ (similarly to how we checked the matching assertion for $\lambda(d 1_G)$ above).  
 From this one sees first, similarly to the last paragraph, that  the range of $\rho$ is
  $\cG^r$, and $\rho$ may be viewed as a completely isometric ($*$-) homomorphism onto $\cG^r$ (with the `reversed product' from
 ${\rm B}(X_G)$). 
  Second, $\lambda$ composed with the inverse of the latter homomorphism  is the map $\mathsf{J} : \cG^\ell \to \cG^r$ defined early in Section \ref{Linstr},
  that is it is the linear extension of the map $\pi_X(g) \to \pi_X^r(g^{-1})$.  This is a $*$-isomorphism $\cG^r \cong \cG^\ell$. 
    \end{proof}

\begin{cor} Suppose that $G$ is finite and $H\le Z(G)$, and that $X$ is a  faithful central $H$-modular operator space.  
The $G$-ification $X_G$  is an operator bimodule over the acting $C^*$-algebra $D_G$, or equivalently is an operator $\cG^\ell$-$\cG^r$-bimodule. 
\end{cor}


\begin{remark}
\begin{enumerate}
\item
Indeed the abstract $G$-ification of $X$ may be alternatively defined 
as a  pair $(V, \Phi)$ where $V$ is an operator $D_G$-bimodule with the restriction of the bimodule action to $H$ central (or equivalently,
$d v = vd$ for $v \in V, d \in D$), 
and $\Phi : X \to V$ is a complete isometry and a $D$-bimodule map
(or if you prefer, an $H$-intertwiner, since $D = {\rm span}(H)$) into $\{ v \in V : a v = v a \; \textrm{for all} \, a \in D_G \}$ such that 
$V = \bigoplus_{gH \,\in\, G/H} \, \Phi(X) \cdot \alpha(g)$.
Indeed by Theorem \ref{ces} such $V$ is a $G$-modular operator space via the action $\alpha$ above, e.g.\ $g v = \alpha(g) \cdot v$ where the latter product is 
the  left $D_G$-module action.  Also 
$a v = v a$ for all $a \in D_G$ if and only if $\alpha(g) \cdot v = v \cdot \alpha(g)$ 
(that is, $gv = vg$) for all $g \in G$. 

Conversely, any abstract $G$-ification in the earlier sense is  $X_G$ by our uniqueness theorem (assuming ergodicity and admissibility).  As we said, 
$X_G$  is an operator $D_G$-bimodule and the canonical map $\Phi = \fE_X : X \to X_G$ satisfies the conditions in the last paragraph. 
In particular $a \fE_X(x)  =\fE_X(x) a$ for all $a \in D_G$,  since this is true for $a \in \alpha(G)$ and $\alpha(G)$ spans $D_G$.

\smallskip

\item
 We show   (in contrast to Example \ref{61}) that for some $H$ we have that  $D_G$ (and hence the algebras  $\cG^\ell$ and $\cG^r$) depend on
  $X$ and $Z(X)$. 
We claim that for $H = \bZ_5$, which
is   part of the admissible pair with $G = \bZ_{25}$, 
the $C^*$-algebra $D_G$ depends on $X$,  
and indeed can have different dimensions 
for different $X$.   To see this note that from the direct sum decomposition  
of $D_G$ we have dim$(D_G) = m \, {\rm dim}(D)$, where $m = |G/H|$.   So it is enough to show that 
dim$(D)$ depends on $X$. 

 Note that $Z_5$ acts ergodically on $X= l^\infty_4(\bC)$ via its representation 
  $$\{ (1,1, 1,1), (z, z^2, z^3, z^4), (z^2,z^4, z, z^3) , \cdots \}$$ (there are five n-tuples here) for a primitive 5th root of unity $z$.   
  This  representation of  $\bZ_5$ generates  a 4 dimensional C*-algebra, since the first four of these 
  vectors constitute an (invertible)  Vandermonde matrix.  Thus
  $D_G$ has a different dimension to what it has for the canonical  representation of $Z_5$  inside 
$X = \bC$ as the 5th roots of unity. 
\end{enumerate}
\end{remark}

\subsection*{Concrete representations of $\cG^\ell$ and $\cG^r$} 
Next, we give a characterization of those $G$-modular representations $(\pi, \rho, \cK)$ of the operator space $G$-ification $X_G$ of a central $H$-modular operator space $X$, for which $C^*_\pi(G)$ is canonically $*$-isomorphic to $\cG^\ell$.

A left-$G$-modular representation $(\pi, \rho, \cK)$ of the operator space $G$-ification $X_G$ is said to be \emph{$H$-monic} if for any $a\in {\rm span}\{\pi(h): h\in H\}$, $a\,\rho(\fE_X(X)) = 0$ implies $a=0$.

\begin{thm}\label{thm:conc-mod}
Let $(\pi, \rho, \cK)$ be a left-$G$-modular representation of the operator space $G$-ification $X_G$. The map $\pi(g)\mapsto \pi_X(g)$ extends to a $*$-isomorphism $C^*_\pi(G)\to \cG^\ell$ if and only if $(\pi, \rho, \cK)$ is $H$-monic.
\begin{proof}
The forward implication follows immediately from the equivariance properties of modular representations.

Conversely, assume $(\pi, \rho, \cK)$ is $H$-monic. If $\sum_{g\in G} c_g \pi(g) = 0$ for some scalars $c_g$, then $\rho\big(\sum_g c_g \pi_X(g)\vf\big) = \sum_g c_g \pi(g)(\rho(\vf)) = 0$ for every $\vf\in X_G$, hence $\sum_g c_g \pi_X(g) = 0$.
This shows that the map $\pi(g)\mapsto \pi_X(g)$ extends to a well-defined linear map $C^*_\pi(G)\to \cG^\ell$, which is a surjective $*$-homomorphism. 

To show that this map is injective, let $\sum_g c_g \pi_X(g) = 0$ for some scalars $c_g$. If $\{ g_i : i = 1, \cdots , m \}$ is a full set of mutually inequivalent representatives of $G/H$, then $$\sum_{i=1}^m \pi(g_i) \big(\sum_{h\in H}c_{g_ih} \pi(h)(\rho(\fE_X(x)))\big)= \rho\big(\sum_g c_g \pi_X(g)(\fE_X(x)) \big)
= 0$$ for every $x\in X$. 
Thus, by Lemma~\ref{lem:pol-id}, $\sum_{h\in H}c_{g_ih} \pi(h)(\rho(\fE_X(x)))= 0$, and it follows $\sum_{h\in H}c_{g_ih} \pi(h)= 0$ for every $i=1, 2, \dots, m$, by $H$-monicity.
Hence, $\sum_{g\in G} c_g \pi(g) = \sum_{i=1}^m \pi(g_i) \big(\sum_{h\in H}c_{g_ih} \pi(h)\big)=0$. This completes the proof.
\end{proof}
\end{thm}

The following lemma guarantees existence of $H$-monic left-$G$-modular representations.

\begin{lem}\label{lem:nondeg-ex}
Let $(\pi, \rho, \cK)$ be a left-$G$-modular representation of the operator space $G$-ification $X_G$. Let $p\in C^*_\pi(H)$ be the maximal projection such that $\rho(\fE_X(X))p = 0$. Then $(1-p)\cK$ reduces both $\pi$ and $\rho$, and the corresponding subrepresentation $(\pi^\prime, \rho^\prime, (1-p)\cK)$ is a $H$-monic left-$G$-modular representation of the operator space $G$-ification $X_G$. 
\begin{proof}
That $(1-p)\cK$ reduces both $\pi$ and $\rho$ follows from the facts that $\pi(H)$ commutes with both $\pi(G)$ and $\rho(X_G)$. Moreover, it follows that $\rho(X_G) p= 0$, and therefore $\rho^\prime$ is a completely isometric representation of $X_G$. 

 If $q$ is a projection in $C^*_{\pi^\prime}(H) = (1-p) C^*_\pi(H)$ such that $q \rho^\prime(\fE_X(X))=0$, then $q (1-p) \leq p$ by maximality of $p$, which forces 
 $q= (1-p)  q = 0$. This implies that $(\pi^\prime|_H, \rho^\prime\circ\fE_X, (1-p)\cK)$ is $H$-monic.
 Indeed  the set $\{a\in C^*_{\pi^\prime}(H) : a\rho^\prime(\fE_X(X))=0\}$ is 
easily argued
to be a $C^*$-algebra, so generated by its projections since it is finite-dimensional.
\end{proof}
\end{lem}

\begin{cor}
For every $a=\sum_{g\in G} c_g\pi_X(g)\in\cG^\ell$ we have
\[
\|a\|_{\cG^\ell} = \inf_{(\pi, \rho, \cK)}\|\sum_{g\in G} c_g\pi(g)\|,
\]
where the infimum is taken over all left-$G$-modular representations $(\pi, \rho, \cK)$ of the operator space $G$-ification $X_G$. This infimum is in fact achieved.
\begin{proof}
Note that in the proof of Theorem~\ref{thm:conc-mod}, for existence of a (surjective) $*$-homomorphism $C^*_\pi(G)\to \cG^\ell$ we did not use $H$-monicity condition. Thus, we have the inequality $\le$ in the above. By Lemma~\ref{lem:nondeg-ex}, there is a left-$G$-modular representation $(\pi, \rho, \cK)$ of $X_G$ that yields equality.
\end{proof}
\end{cor}

\section{Characterization of $G$-ified spaces}\label{sec:Gified}
A very natural question now is that given an admissible $H\le Z(G)$ and a $G$-modular operator space $Y$ that is $H$-central and $H$-ergodic, is there a central ergodic $H$-modular operator space $X$ such that $Y=X_G$? 

The answer is in general no, even already in the complexification case. Indeed, a complex $C^*$-algebra $A$ is the complexification of a real $C^*$-algebra if and only if $A$ is isomorphic to its opposite $C^*$-algebra $A^{\rm op}$. A celebrated result of Connes~\cite{Con75} provides examples of von Neumann algebras that do not satisfy this.

In contrast, we prove that for certain pairs $H\le G$, every $G$-modular real operator space is the $G$-ification of an $H$-modular real operator space. These examples include quaternionic operator spaces, as well as 
$D_4$-modular real operator spaces.
This is a consequence of a characterization result for $G$-modular operator spaces $Y$ that are $G$-ification of some $H$-modular operator space $X$, which we prove in this section.

\begin{defn}
Let $H\le Z(G)$. A subgroup $\G\le \Aut_H(G)$ is said to be \emph{balanced} (for $H\le G$) if $(G, H)$ is $\G$-admissible, ${\rm Inn}(G)\subseteq \G$, and for every $\sigma \in \G$ the set $L_\sigma:=\{g^{-1}\sigma(g) \,:\, g\in G\}$ is a subgroup of $G$. 
\end{defn}

When $G$ is abelian, then the last condition above obviously holds. In particular, in this case if $H$ is $\G$-admissible for any $\G\le \Aut_H(G)$, then $\Aut_H(G)$ is balanced.
Also, $\G = {\rm Inn}(Q_8)$ is balanced for $\bZ_2\le Q_8$.

\begin{defn}\label{def:align}
Let $H\le Z(G)$, and let $\G\le \Aut_H(G)$ be balanced for $H\le G$.
A $G$-modular operator space $Y$ is said to be \emph{$(G, H, \G)$-aligned} if 
\begin{itemize}
\item[(i)\ ]
it is ergodic for the restriction of the action to $H$ and $L_\sigma$ for every non-trivial $\sigma\in\G$;
\item[(ii)\,]
there is an action of $\G$ on $Y$ by complete isometries such that $\sigma(gy) = \sigma(g) \sigma(y)$ for all $\sigma\in \G$, $g\in G$ and $y\in Y$;
\item[(iii)]
for every $g\in G$ and $y\in Y$, $\sigma_g(y)= gyg^{-1}$, where $\sigma_g$ is the inner automorphism of $G$ 
(necessarily in $\Gamma$) defined by $g$, and $\sigma_g(y)$ is its action on $Y$ from (ii).
\end{itemize}
\end{defn}

{\bf Remark.}  For the reader curious about the bias towards the left action of $G$ in condition (ii) above, we note that the conclusion of the next theorem will imply that we also have $\sigma(y g) = \sigma(y) \sigma(g)$.

\begin{thm}\label{thm:gified-charac}
Let $G$ be finite, $H\le Z(G)$, and $Y$ is a $G$-modular operator space. If $Y$ is $(G, H, \G)$-aligned for some balanced $\G\le\Aut_H(G)$, then there is a central $H$-modular operator space $X$ such that $Y=X_G$.
\end{thm}

\begin{proof}
Define $E_0:Y\to Y$ by $E_0(y) :=\frac{1}{|\G|}\sum_{\sigma \in \G} \sigma(y)$ for all $y\in Y$. 
Observe that $E_0$ is an idempotent, and for every $g\notin H$ we have
\begin{align*}
|\G|^2 E_0(gE_0(y)) 
&= 
\sum_{\tau \in \G} \tau\big(g\sum_{\sigma \in \G} \sigma(y)\big)
= \sum_{\tau \in \G}\sum_{\sigma \in \G} \tau\big(g\sigma(y)\big)
\\&= \sum_{\tau \in \G}\tau(g)\sum_{\sigma \in \G} \tau\sigma(y)
= \sum_{\tau \in \G}\tau(g)\sum_{\sigma \in \G}\sigma(y) .
\end{align*}
Similarly to Lemma~\ref{lem:sumzero}, it follows $\sum_{\tau \in \G} \tau(g) =0$, thus $E_0(gE_0(y)) =0$.

Let $\{ g_i : i = 1, \cdots , m \}$ be a full set of mutually inequivalent representatives of the coset space $G/H$ with $g_1=1$.
We show that $y = \sum_{i=1}^m g_iE_0(g_i^{-1}y)$ for every $y\in Y$.
For this, let $z:= y - \sum_{i=1}^m g_iE_0(g_i^{-1}y)$. Then for every $j = 1, \cdots , m$,
\begin{align*}
E_0(g_j^{-1}z)
&= E_0(g_j^{-1}y) - \sum_{i=1}^m E_0(g_j^{-1}g_iE_0(g_i^{-1}y))
\\&= E_0(g_j^{-1}y) - E_0(E_0(g_i^{-1}y))
= 0 .
\end{align*}
It follows $\sum_{\sigma\in\G}\sum_{j=1}^m g_j\sigma(g_j^{-1})\sigma(z) = |\G|\sum_{j=1}^m g_jE_0(g_j^{-1}z) = 0$.
Observe that for every $\sigma\in\G$, the list $(g_j\sigma(g_j^{-1}))_{j=1}^m$ consists of elements of the subgroup $L_\sigma$, each of which appears exactly $[G_\sigma:H]$ times, where $G_\sigma=\{g\in G :\sigma(g)=g\}$.
For every non-trivial $\sigma\in\G$, since $Y$ is $L_\sigma$-ergodic,  
it follows 
$\sum_{j=1}^m g_j\sigma(g_j^{-1})\sigma(z) = [G_\sigma:H] \sum_{k\in L_\sigma} k\sigma(z) = 0$.
This implies, $z = \frac{1}{m}\sum_{\sigma\in\G}\sum_{j=1}^m g_j\sigma(g_j^{-1})\sigma(z) = 0$. 

Let $X:=E_0(Y)$ be the image of $E_0$, which is the fixed point subspace of $\G$-action on $Y$. 
The above shows $Y = \sum_{j=1}^m g_j X$. Then arguing similarly to Lemma~\ref{lem:pol-id}, we get $Y = \oplus_{j=1}^m g_j X$.

By part (iii) of Definition~\ref{def:align}, for every $g\in G$ and $x\in X$, $gxg^{-1} = \sigma_g(x) = x$, which implies that $X$ is contained in the subspace $\{y\in Y : gy=yg\}$.

For $\sigma\in \G$, $x\in X$ and $h\in H$, $\sigma(hx) = \sigma(h) \sigma(x) = hx$, which implies $hx\in X$, and so $X$ is a central $H$-modular operator space.

Thus, denoting by $\iota: X\to Y$ the inclusion map, the pair $(Y, \iota)$ is an abstract $G$-ification of $X$, which is reasonable by part (ii) of Definition~\ref{def:align}. Hence, by Theorem~\ref{thm:Gif-unqabs}, $Y=X_G$ is the reasonable operator space $G$-ification of $X$.
\end{proof}

Note that the canonical reasonable $G$-ification of any ergodic central $H$-modular operator space $X$ satisfies conditions (ii) 
and (iii) of Definition~\ref{def:align} for $\G=\Aut_H(G)$. If $H$ is $\G$-admissible for a $\G\le \Aut_H(G)$ such that $L_\sigma=H$ for every $\sigma \in \G$, then the the canonical reasonable representation also satisfies condition (i) of Definition~\ref{def:align}.
In particular, in this case a $G$-modular operator space $Y$ is the $G$-ification of an $H$-modular operator space $X$ if and only if $Y$ is $(G, H, \G)$-aligned. 

In the case of complexification, $L_\sigma=\bZ_2$ for every $\sigma\in\Aut_{\bZ_2}(\bZ_4)$. An action of $\Aut_{\bZ_2}(\bZ_4)$ on a real 
$\bZ_2$-modular operator space $Y$ is just a degree-2 bijective completely isometric real linear $\theta$ on $Y$, where condition (ii) of Definition~\ref{def:align} translates to $\theta(iy) = -i\theta(y)$. In particular, we recover the well-known fact that a complex operator space $Y$ is the reasonable complexification of a real space $X$ if and only if there exists a conjugate linear degree-2 bijective real complete isometry $\theta$ on $Y$.

Given a group $G$, we continue to use the notation $\sigma_g$ for the inner automorphism of $G$ defined by the element $g\in G$.

\begin{thm}\label{thm:all-G-if}
Let $G$ be finite and $H\le Z(G)$. If ${\rm Inn}(G)$ is balanced for $H\le G$, then every $G$-modular $H$-ergodic real operator space $Y$ is the reasonable $G$-ification of a central $H$-modular real operator space $X$.
\begin{proof}
We show that $Y$ is $(G, H, \G)$-aligned for $\G={\rm Inn}(G)$, hence the result follows by Theorem~\ref{thm:gified-charac}.

For every $g\notin Z(G)$, we have $\{k^{-1}\sigma_g(k) : k\in G\} = \{g\sigma_{k^{-1}}(g^{-1}) : k\in G\}^{-1}$ includes $H$ by the admissibility assumption. 
Thus, condition (i) of Definition~\ref{def:align} is satisfied by $H$-ergodicity.

Define the completely isometric action of $\G$ on $Y$ by $\sigma_g(y) = gyg^{-1}$ for every $g\in G$.  
 For every $g, k\in G$ and $y\in Y$, we have $$\sigma_g(ky) = gkyg^{-1} = gkg^{-1}gyg^{-1} = \sigma_g(k)\sigma_g(y).$$ This is condition
  (ii) of Definition~\ref{def:align}, and (iii) holds by definition of the $\G$-action.
\end{proof}
\end{thm}

\begin{cor}
Every $Q_8$-modular real operator space $Y$ is the reasonable quaternification of a real operator space $X$. 
\begin{proof}
The change-of-sign action of $H = \bZ_2$ on any real space is ergodic, and as we remarked after Definition~\ref{def:align}, 
$\Gamma = {\rm Inn}(Q_8)$ is balanced for $\bZ_2\le Q_8$. Hence, the result follows immediately from Theorem~\ref{thm:all-G-if}.
\end{proof}
\end{cor}

One may  define a {\em quaternionic operator space} to be a   $Q_8$-modular real operator space with the canonical left and right action of $\bH$ as in e.g.\ 
Example \ref{61}.  The last result then says that, remarkably, every quaternionic operator space is the quaternification of a real operator space.   (We point out that 
similar sounding statements were shown in \cite{Ng} for the categories of quaternionic $C^*$-algebras, quaternionic Hilbert spaces etc.) 
However a similar phenomenon holds in very many other
situations.  For example:

\begin{cor}
Every $D_4$-modular real operator space $Y$ is the reasonable $D_4$-ification of a real operator space $X$. 
\begin{proof}
One can readily check that ${\rm Inn}(D_4)$ is balanced for $\bZ_2= Z(D_4)\le D_4$. 
Again, since the change-of-sign action of $\bZ_2$ on any real space is ergodic, the result follows from Theorem~\ref{thm:all-G-if},
as in the last proof.
\end{proof}
\end{cor}


\section{Examples}\label{sec:ex}

Let $H=\bZ_2 \cong \{ \pm 1 \}$, acting on $X$ as `change of sign'.   In this case $G(\bR) \subset M_m(\bR)$  is an $m$ dimensional real $C^*$-algebra 
(where $m = |G|/2$), and $X_G \cong X \otimes_{\rm min} G(\bR)$.   Nearly all the examples below are of this latter form.    We leave the details in these examples to the reader.  Computing the matrix form of $X_G$ amounts to using the group multiplication table to write each  $g_i g_j^{-1}$ as $h g_k$, for  a complete set $\{ g_n \}$ of representatives for the coset space $G/H$ with $g_1 = 1$, and $h \in H$.   See Section \ref{marp}.

\subsection{Complexification}
Let $H=\bZ_2$ and $G=\bZ_4$. As seen in Example~\ref{ex:comp}, if $X$ is a real operator space and $\bZ_2\act X$ by change of sign $x\mapsto -x$, then $X_G = X_c$ is the complexification of $X$.
 The reasonable operator space norm on $X_G = X_c$ is defined via its matrix form (see Section \ref{marp})
$$\Big\{ \begin{bmatrix}
       a    & -b \\
       b & a 
\end{bmatrix} \,:\, a,b \in X \Big\}.$$ 

\subsection{Quaternification} \label{ExQ} 
Let $H=\bZ_2$ and $G=Q_8$. As seen in Example~\ref{ex:quat}, if $X$ is a real operator space and $\bZ_2\act X$ by change of sign $x\mapsto -x$, then $X_G = X_\bH$ is the quaternification of $X$.
 The reasonable operator space norm on $X_G = X_\bH$ is defined via its matrix form (see Section \ref{marp})
$$\Big\{ \begin{bmatrix}
       a    & -b & -c & -d \\
       b & a & -d  & c \\
      c &  d  & a    & -b \\
      d  & -c & b & a
\end{bmatrix} \,:\, a,b,c,d \in X \Big\}.$$
In this case $\bR_G = \bH$ is a real $C^*$-algebra, real operator space, and real Hilbert space with its natural norm and real inner product (isometric to $l^2_4(\bR$)).

\subsection{Further examples}      
Let $H=\bZ_2 \cong \{ 1, r^2 \}$ where $r$ has order $4$ in $D_4$, and $G=D_4$. Consider the action of $\bZ_2$ on $X$ defined by $x\mapsto -x$. Then $X_G$ has matrix form (see Section \ref{marp})
$$\Big\{ \begin{bmatrix}
       a    & -b & c & d \\
       b & a & d  & -c \\
      c &  d  & a    & -b \\
      d  & -c & b & a
\end{bmatrix} \,:\, a,b,c,d \in X \Big\}.$$ 
  Note how similar this is to the matrix form for the quaternions. 
    
Similarly for the pair $(\bZ_8,\bZ_2)$, $X_G$ has matrix form 
$$\Big\{ \begin{bmatrix}
       a    & -d & -c & -b \\
       b & a & -d  & -c \\
      c &  b  & a    & -d \\
      d  & c & b & a
\end{bmatrix} \,:\, a,b,c,d \in X \Big\},$$
which is (not obviously)  completely isometric to the matrix space obtained by induction in the stages $(\bZ_8,\bZ_4)$ and $(\bZ_4,\bZ_2)$.
    
For the nonadmissible pair $(\bZ_6,\bZ_2)$ on the other hand, $X_G$ has matrix form 
    $$\Big\{ \begin{bmatrix}
       a    & -c &  -b \\
       b & a & -c \\
      c &  b  & a   
    \end{bmatrix} \,:\, a,b,c \in X \Big\}.$$  The subspace corresponding to setting $b = -c$ is the fixed point space of the two-member 
group Aut$_H(G)$, from which it is easy to see that  this is not a reasonable G-fication of its fixed point space.  It is also easy to see from this that several of our earlier lemmas and propositions requiring $(G,H)$ admissible fail for this example.

For a (real or complex) Hilbert  space the $G$-ification  is not a Hilbert space  isometrically except in a few known cases.  For example if $\cH$ is one dimensional and $H = \{ 1, -1 \}$ then $H^{\rm col}_G$  is a unital finite dimensional
$C^*$-algebra.   The only such that are also isometrically Hilbert spaces are known to be isomorphic to $\bR, \bC,$ and the quaternions \cite{Ing}. 
The complexification or quaternification of a real Hilbert column space is a Hilbert space.  More generally, it seems possible that
for example for certain 
groups generated by anticommuting anti-symmetric unitaries,  the $G$-ification of a Hilbert column space may yield new operator spaces of interest in 
quantum physics.  This is related to the subject of spin systems, and anticommutator relations 
in quantum physics (using ideas in the Remark on p.\ 174 in \cite{Pisbk}, multiplying those unitaries by $i$).

\end{document}